\documentclass[11 pt]{amsart}

\usepackage[utf8]{inputenc}
\usepackage{tikz, tikz-3dplot}
\usepackage{amsmath, amsthm, amssymb,graphics }
\usetikzlibrary{arrows,shapes,positioning}
\usetikzlibrary{intersections}
\usetikzlibrary{decorations.pathreplacing,calligraphy}

\usepackage{amsmath,amsthm,amssymb,amscd,color, xcolor,mathtools,url,tikz}
\usepackage{bbm}
\usepackage{pgfplots}
\pgfplotsset{compat=newest}
\usepgfplotslibrary{fillbetween}

\newtheorem{thm}{Theorem}[section]
 \newtheorem{cor}[thm]{Corollary}
 \newtheorem{lem}[thm]{Lemma}
 \newtheorem{prop}[thm]{Proposition}

\newtheorem{conj}{Conjecture}

 \theoremstyle{definition}
 
 \theoremstyle{remark}
 \newtheorem{rem}[thm]{Remark}
 
 \numberwithin{equation}{section}

\def\les{{\;\lessapprox}\;}\def\ges{{\;\gtrapprox}\;}
\def\be#1 {\begin{equation} \label{#1}}
\def\ee{\end{equation}}

\def\sqw{\hbox{\rlap{\leavevmode\raise.3ex\hbox{$\sqcap$}}$%
\sqcup$}}
\def\findem{\ifmmode\sqw\else{\ifhmode\unskip\fi\nobreak\hfil
\penalty50\hskip1em\null\nobreak\hfil\sqw
\parfillskip=0pt\finalhyphendemerits=0\endgraf}\fi}

\newcommand{\R}{\mathbb R}

\newcommand{\C}{\mathbb C}
\newcommand{\s}{\mathcal S}

\author[C. Demeter]{Ciprian Demeter}
\address{Ciprian Demeter, Department of Mathematics, 831 E Third Street, Bloomington, Indiana, U.S.A.}
\email{demeterc@iu.edu}

\author[P. Germain]{Pierre Germain}
\address{Pierre Germain, Department of Mathematics, Huxley Building, South Kensington Campus,
Imperial College London, London SW7 2AZ, United Kingdom}
\email{pgermain@ic.ac.uk}

\title{$L^2$ to $L^p$ bounds for spectral projectors on the Euclidean two-dimensional torus}



\begin{document}

\begin{abstract} We consider spectral projectors associated to the Euclidean Laplacian on the two-dimensional torus, in the case where the spectral window is narrow. Bounds for their $L^2$ to $L^p$ operator norm are derived, extending the classical result of Sogge; a new question on the convolution kernel of the projector is introduced. The methods employed include $\ell^2$ decoupling, small cap decoupling, and estimates of exponential sums. 
\end{abstract}

\maketitle

\tableofcontents 

\section{Introduction}

\subsection{Eigenfunctions of the Laplacian on the two-dimensional torus}

We consider the torus
$$
\mathbb{T}^2 = \mathbb{R}^2 / \mathbb{Z}^2,
$$
on which Fourier series are given by
$$
f(x) = \sum_{k \in \mathbb{Z}^2} \widehat{f}_k e^{2\pi i k \cdot x}, \qquad \widehat{f}_k = \int_{\mathbb{T}^2} f(x) e^{-2\pi i k \cdot x} \,dx.
$$

A classical question is to estimate the $L^p$ norms of eigenfunctions of the Laplacian: if $\varphi$ is such that $- \Delta \varphi = \lambda^2 \varphi$ on $\mathbb{T}^2$, and if it is normalized in $L^2$, what is the optimal bound on $\| \varphi \|_{L^p}$, for $p \geq 2$? What should be expected is unclear (the question is asked in~\cite{Bourgain1}), but one possibility is that
\begin{equation}
\label{macareux}
\| \varphi \|_{L^p} \lesssim_p 1 \quad \mbox{if $p<\infty$, while} \quad \| \varphi \|_{L^\infty} \lesssim_\epsilon \lambda^\epsilon.
\end{equation}
This was proved for $p=4$ by Cooke~\cite{Cooke} and Zygmund~\cite{Zygmund}; and for $p=\infty$, the bound $\lambda^{\frac{C}{\log \log \lambda}}$ follows from the divisor bound in Gaussian integers.
Proving optimal bounds for $\| \varphi \|_{L^p}$ for any $p \in (4,\infty)$ appears to be a very hard problem, which can be relaxed by considering spectral projectors on narrow spectral windows, to which we now turn.

\subsection{A conjecture on spectral projectors on narrow windows} For $\lambda>2$ and $\delta<1$, 
the spectral projector on the range $(\lambda - \delta, \lambda + \delta)$ for the square root of the Euclidean Laplacian is given through functional calculus by the formula
\begin{align*}
P_{\lambda,\delta} = \mathbf{1}_{(\lambda - \delta,\lambda + \delta)} (\sqrt{-\Delta}) \qquad \mbox{or} \qquad P_{\lambda,\delta} f (x) = \sum_{k \in \mathcal{A}_{\lambda ,\delta}} \widehat{f}_k e^{2\pi i k \cdot x},
\end{align*}
where $\mathcal{A}_{\lambda,\delta}$ is the annulus with inner radius $\lambda-\delta$ and width $2 \delta$:
$$
\mathcal{A}_{\lambda,\delta} = \{ x \in \mathbb{R}^2, \, \lambda-\delta < |x| <\lambda + \delta \}.
$$
Two consecutive eigenvalues of $\sqrt{-\Delta}$ close to $\lambda$ are at least $\sim \frac{1}{2} \lambda^{-1}$ apart. Thus, if $\delta = \frac{1}{4} \lambda^{-1}$, bounding $P_{\lambda,\delta}$ is equivalent to bounding eigenfunctions of the Laplacian.

In the present paper, we consider the following conjecture, which focuses on the case where the spectral window is at least slightly larger than $\lambda^{-1}$.

\begin{conj}[\cite{GermainMyerson1}] If $p \geq 2$, the operator norm of $P_{\lambda,\delta}$ satisfies for any $\kappa>0$
\label{conjproj}
\begin{equation}
\label{chardonneret}
\| P_{\lambda,\delta} \|_{L^{2} \to L^p} \lesssim_{\kappa,p}  \lambda^{\frac{1}{2} - \frac{2}{p}} \delta^{\frac 12} + (\lambda \delta)^{\frac{1}{4} - \frac{1}{2p}}  \qquad \mbox{if} \quad \delta > \lambda^{-1+\kappa}
\end{equation}
or in other words
$$
\| P_{\lambda,\delta} \|_{L^{2} \to L^p}
\lesssim_{\kappa,p} \left\{
\begin{array}{ll}
(\lambda \delta)^{\frac{1}{4} - \frac{1}{2p}} & \mbox{if $p \leq 6$} \\
(\lambda \delta)^{\frac{1}{4} - \frac{1}{2p}} & \mbox{if $p \geq 6$, $\delta \leq \lambda^{-1+\frac{8}{p+2}}$} \\
\lambda^{\frac{1}{2} - \frac{2}{p}} \delta^{\frac 1 2} & \mbox{if $p \geq 6$, $\delta \geq \lambda^{-1+\frac{8}{p+2}}$}.
\end{array}
\right.
$$
The conjecture is said to be satisfied with $\epsilon$ loss if
\begin{equation}
\| P_{\lambda,\delta} \|_{L^{2} \to L^p} \lesssim_{\kappa,p,\epsilon} \lambda^\epsilon \left[  \lambda^{\frac{1}{2} - \frac{2}{p}} \delta^{\frac 12} + (\lambda \delta)^{\frac{1}{4} - \frac{1}{2p}}  \right]\qquad \mbox{if} \quad \delta > \lambda^{-1+\kappa}. 
\end{equation}
\end{conj}

\begin{rem} The justification for this conjecture can be found in~\cite{GermainMyerson1}, where two basic examples are considered: the Knapp example, and the spherical example. They lead to the two terms on the right-hand side of~\eqref{chardonneret}.
\end{rem}

\begin{rem} Combining the conjecture with the guess~\eqref{macareux}, it might be the case that the estimate~\eqref{chardonneret} is true for any $\kappa \geq 0$ and $p \in [2, \infty]$, with an implicit constant $C(p,\kappa)$ which only blows up as $(p,\kappa) \to (\infty,0)$.
\end{rem}

\subsection{Known results on Conjecture~\ref{conjproj}}
Conjecture~\ref{conjproj} is known to hold in a number of cases:
\begin{itemize}
\item If $\delta =1$, the conjecture corresponds to the fundamental result of Sogge~\cite{Sogge}, which holds on any Riemannian manifold (see also~\cite{BlairHuangSogge} for a recent extension to logarithmically small spectral windows for general nonpositively curved manifolds, including in particular the torus).
\item If $p=4$, the conjecture was proved for the full range $\lambda^{-1} <\delta<1$ by Bourgain-Burq-Zworski~\cite{BourgainBurqZworski}.
\item If $p \leq 6$, the conjecture with $\epsilon$ loss is a consequence of the $\ell^2$ decoupling of Bourgain-Demeter~\cite{BourgainDemeter3} as was observed in~\cite{GermainMyerson1}.
\item If $p = \infty$, the conjecture for the full range $\lambda^{-1} <\delta<1$ with $\epsilon$ loss follows immediately from the  bound $\lambda^\epsilon$ for the $L^\infty$ norm of eigenfunctions.
\item If $p=\infty$, the conjecture without $\epsilon$ loss would be a consequence of the estimate $N(\lambda) = \pi \lambda^2 + O(\lambda \delta)$ for the number $N(r)$ of integer points in the disc with radius $r$. This corresponds to the Gauss circle problem, for which the best current bound, due to Huxley~\cite{Huxley2}, allows $\delta > \lambda^{-\frac{77}{208} + \epsilon}$, with $\frac{77}{208} \sim 0.37$. Note however that Conjecture~\ref{conjproj} is expected to hold down to $\delta = \lambda^{-1+\kappa}$, while the the estimate $N(\lambda) = \pi \lambda^2 + O(\lambda \delta)$ can only be true for $\delta > \lambda^{-\frac 12}$
\end{itemize}

For the two-dimensional Euclidean cylinder, the conjecture is identical, and it has been proved with $\epsilon$ loss~\cite{GermainMyerson2}. Finally, this conjecture has also been considered in higher dimensions, for which we refer to \cite{Bourgain1,Bourgain2,BourgainDemeter1,BourgainDemeter2,BourgainDemeter3,BourgainShaoSoggeYao,GermainMyerson1,GermainMyerson3,Hickman}.

\subsection{A new conjecture}

The convolution kernel
$$
\Phi_{\lambda,\delta} = \sum_{k \in \mathcal{A}_{\lambda,\delta} \cap \mathbb{Z}^2} e^{2\pi i k \cdot x} \quad \mbox{is such that} \quad P_{\lambda,\delta} f = \Phi_{\lambda,\delta} * f.
$$

\begin{conj}
\label{conjkernel}
If $p \geq 2$ and $\kappa > 0$, then if $\delta > \lambda^{-1 + \kappa}$,
$$
\left\| \Phi_{\lambda,\delta} \right\|_{L^p} \lesssim_{p,\kappa} \lambda^{1 - \frac 2 p} \delta + (\lambda \delta)^{\frac 12}
$$
or in other words,
$$
\left\| \Phi_{\lambda,\delta} \right\|_{L^p} \lesssim 
\left\{
\begin{array}{ll}
(\lambda \delta)^{\frac 12} & \mbox{if $2 \leq p \leq 4$}\\
(\lambda \delta)^{\frac 12} & \mbox{if $p \geq 4$ and $\delta < \lambda^{\frac 4 p -1}$}\\
\lambda^{1 - \frac 2 p} \delta & \mbox{if $p \geq 4$ and $\delta > \lambda^{\frac 4 p -1}$}.
\end{array}
\right.
$$
\end{conj}

This conjecture is interesting in two respects: first, it is partially equivalent to Conjecture~\ref{conjproj}; and second, it is equivalent to questions on additive energies of subsets of $\mathbb{Z}^2$. See Section~\ref{sectionconvolution} for more on these two points. This conjecture is based on the two following observations, which also show that the conjecture is optimal, if true.
\begin{itemize}
\item A naive counting argument shows that, for a given $\delta$, there exists in any interval of length $1$ a $\lambda$ such that $\# \mathcal{A}_{\lambda,\delta} \cap \mathbb{Z}^2 \gtrsim \lambda \delta$. For this choice of $\delta$ and $\lambda$, H\"older's inequality and Parseval's theorem imply that
$$
\| \Phi_{\lambda,\delta} \|_{L^p} \gtrsim \| \Phi_{\lambda,\delta} \|_{L^2} = \left( \# \mathcal{A}_{\lambda,\delta}' \right)^{\frac 12} \gtrsim (\lambda \delta)^{\frac 12}.
$$
\item Still considering $\lambda$ and $\delta$ such that $\# \mathcal{A}_{\lambda,\delta}  \cap \mathbb{Z}^2 \gtrsim \lambda \delta$, Bernstein's inequality gives that
$$
\| \Phi_{\lambda,\delta} \|_{L^p} \gtrsim \lambda^{-\frac{2}{p}} \| \Phi_{\lambda,\delta} \|_{L^\infty} = \lambda^{-\frac 2 p} \# \mathcal{A}_{\lambda,\delta}' \gtrsim \lambda^{1 - \frac{2}{p}} \delta.
$$
\end{itemize}

Note that the conjecture cannot hold all the way to $\kappa =0$ and $p \geq 2$, since it would imply a uniform bound on the number of lattice points on a circle, which is known to fail.

\subsection{Main results} Our main results verify the conjectures for various ranges in $(p,\lambda,\delta).$

\begin{thm} \label{mainthmA} \begin{itemize}
\item[(i)] Conjecture~\ref{conjproj} holds if $2 \leq p <6$, or $p \geq6$ and $\delta > \min \left( \lambda^{- \frac{1 - \frac{6}{p}}{3 - \frac 2p}} , \lambda^{- \frac{10 - \frac{64}{p}}{29 - \frac{14}{p}}+\epsilon} \right)$.

\item[(ii)]
Conjecture~\ref{conjproj} holds with $\epsilon$ loss if $2 \leq p \leq 10$, or $\delta > \lambda^{-\frac 13}$, or $p=\infty$.
\end{itemize}
\end{thm}

This statement follows from combining Theorem~\ref{thm1} and Theorem~\ref{thm2}.
Turning to Conjecture~\ref{conjkernel}, it is a consequence of Conjecture~\ref{conjproj} if $\delta > \lambda^{-1+\frac{8}{p+2}}$; thus, the previous theorem gives the validity of Conjecture~\ref{conjkernel} in some range. Furthermore, combining Corollary~\ref{dd13} and Proposition~\ref{dd15} gives the following theorem.

\begin{thm} \label{mainthmB} Conjecture~\ref{conjkernel} holds with $\epsilon$ loss if $2 \leq p \leq 6$ or $\lambda> \delta^{-\frac 13}$.
\end{thm}

Finally, we refer to Section~\ref{sectiongraphical} for a graphical representation of the ranges in $(p,\lambda,\delta)$.

\subsection{Ideas of the proofs and plan of the paper} \subsubsection{Decomposition into caps} This is the first possible line of attack on Conjecture~\ref{conjproj}, which is carried out in Section~\ref{sectioncaps}; here, caps are rectangles which optimally cover the annulus $\mathcal{A}_{\lambda,\delta}$. We investigate estimates on functions whose Fourier support is restricted to caps containing a bounded number of lattice points, relying crucially on $\ell^2$ decoupling. Combining these estimates with an estimate on caps with a given number of lattice points leads to our result in that section.

\subsubsection{Dyadic decomposition of the kernel} This approach to Conjecture~\ref{conjproj} is carried out in Section~\ref{sectiondyadic}. It relies on a dyadic decomposition of the convolution kernel of $P_{\lambda,\delta}$, which is reminiscent of the original proof of the Stein-Tomas theorem. In the regime where this approach is useful ($p$ large), an important threshold is the line $\delta = \lambda^{-\frac 13}$. Reaching (slightly) smaller $\delta$ can be achieved with the help of pointwise bounds on exponential sums, whose investigation is a classical topic in analytic number theory.

\subsubsection{Conjecture~\ref{conjkernel} and small caps} Section~\ref{sectionconvolution} is dedicated to Conjecture~\ref{conjkernel}. We show that it is partially equivalent to Conjecture~\ref{conjproj}, and explain its connection with additive combinatorics. In order to make progress on this conjecture for $p\leq 6$,  $\ell^2$ decoupling is not strong enough, the right tool proves to be  small cap decoupling. Combining this tool with careful estimates on the distribution of lattice points in $\mathcal{A}_{\lambda,\delta}$ leads to our main result in that section.

\subsection{General curves} \label{subsectiongeneral} How much do our results depend on the geometry of the circle? Is it possible to replace it by an ellipse (which would correspond to general tori $\mathbb{R}^2 / [ \mathbb{Z} e_1 + \mathbb{Z} e_2 ]$, where $e_1$ and $e_2$ are non-colinear vectors in $\mathbb{R}^2$), or even by a general curve? We consider a smooth arc or a closed smooth curve, which is denoted by $\Gamma$ and is compact; the natural generalizations of $P_{\lambda,\delta}$ and $\Phi_{\lambda,\delta}$ are given by
$$
\widetilde{P_{\lambda,\delta}} f = \sum_{k \in \mathcal{N}_\delta(\lambda \Gamma) \cap \mathbb{Z}^2}\widehat{f_k} e^{2\pi i k\cdot x} \quad \mbox{and} \quad \widetilde{\Phi_{\lambda,\delta}}(x) = \sum_{k \in \mathcal{N}_\delta(\lambda \Gamma) \cap \mathbb{Z}^2} e^{2\pi i k \cdot x},
$$
where $\mathcal{N}_\delta(\lambda \Gamma)$ stands for the $\delta$-neighborhood of $\lambda \Gamma$.

Nearly all the proofs below remain valid as long as the curvature of $\Gamma$ does not vanish, and almost all the intermediary estimates proved in this paper still hold. This is the case for the cap counting Lemma~\ref{lemmacounting2}, the $L^4$ estimate Lemma~\ref{lemmaL4}, the decoupling estimate Lemma~\ref{lemmadecoupling}, the exponential sum estimates\footnote{To see why these exponential sum estimates still hold, observe that the asymptotics of the Fourier transform of the superficial measure supported on $\Gamma$ are very similar to~\eqref{asymptoticsJ}. The only difference is that the phase function $|\xi|$ is replaced by a function $\phi(\xi)$, but it remains smooth and 1-homogeneous, see~\cite{Stein}, page 360. } in Section \ref{sectiondyadic}, and the small cap estimates in Section~\ref{sectionprogress}.

But the bound on the number of lattice points in $\mathcal{A}_{\lambda,\delta}$ is lost! For the square torus $\mathbb{R}^2 / \mathbb{Z}^2$, the divisor bound immediately gives the estimate $\lambda^{1+\epsilon} \delta$; but for a general torus, or a general curve, such an argument is not available to estimate $\# \mathcal{N}_\delta(\lambda \Gamma) \cap \mathbb{Z}^2$, and this bound is much more difficult to obtain. Exponential sum bounds seem to be the only possibility, which are closely related, if not equivalent, to the Gauss circle problem. They give the expected result for $\delta > \lambda^{-\frac 13}$ easily, but it is difficult to go significantly below this barrier.

As a result, for general smooth curves (with curvature) and the associated $\widetilde{P_{\lambda,\delta}}$ operators,
\begin{itemize}
\item Theorem~\ref{thm1} remains true for $\delta > \lambda^{-\frac 13}$, and a little beyond, even though we will not compute here the exact exponents.
\item Theorem~\ref{thm2} remains true.
\item Theorem~\ref{dd12} and its corollaries remain true for $\delta > \lambda^{-\frac 13}$.
\end{itemize}
However, conjectures~\ref{conjproj} and~\ref{conjkernel} break down for general curves (with curvature) for small $\delta$: see Remark~\ref{remarkparabola} on the case of the parabola.

\subsection*{Acknowledgement} When working on this article, CD was supported by the NSF grant DMS-2055156. He thanks his student Hongki Jung for help with Figure 1.   PG was supported by a start-up grant from Imperial College and a Wolfson fellowship. He is most grateful to Simon Myerson for sharing his number theoretical insight in numerous conversations. 

\section{Caps decomposition of the kernel}

\label{sectioncaps}

In this section, we prove Conjecture~\ref{conjproj} for some range of $(p,\lambda,\delta)$ by decomposing a function supported on the annulus (in Fourier) into a sum of functions supported on caps (in Fourier).

\subsection{Counting points and caps}

\label{sectioncounting}

Recall that the annulus of inner radius $\lambda-\delta$ and width $2\delta$ is denoted
$$
\mathcal{A}_{\lambda,\delta} = \{ x \in \mathbb{R}^2, \, \lambda - \delta < |x| < \lambda + \delta \}.
$$

Next, we will split the annulus into caps: $\mathcal{A}_{\lambda,\delta}$ can be covered by a finitely disjoint collection $\mathcal{C}$ of caps $\theta$, of dimensions $\delta \times (\lambda \delta)^{\frac 12}$:
$$
\mathcal{A}_{\lambda,\delta} \subset \cup_{\theta \in \mathcal{C}} \theta.
$$
The number of such caps is $\sim \lambda^{\frac 12} \delta^{-\frac 12}$. We will be interested in the number of lattice points contained in a cap; therefore the notation
$$
\theta' = \theta \cap \mathbb{Z}^2, \qquad \mbox{and, more generally,}\qquad E' = E \cap \mathbb{Z}^2 \;\; \mbox{if $E \subset \mathbb{R}^2$}
$$
will be useful.
The maximal cardinality of $\theta'$ is $\sim \lambda^{\frac{1}{2}} \delta^{\frac 12}$; as for the average cardinality of $\theta'$, it is expected to be given by the area of $\theta$, namely $\lambda^{\frac 1 2} \delta^{\frac 3 2}$.

We will now split the collection $\mathcal C$ into $\cup_s \mathcal{C}_s \cup \mathcal{C}_0$ as follows
\begin{itemize}
\item If $2^s$ is such that $1 + \lambda^{\frac{1}{2}} \delta^{\frac 32} \ll 2^s \lesssim \lambda^{\frac 12} \delta^{\frac 12}$, $\mathcal{C}_s$ gathers all caps $\theta$ such that $\#  \theta' \sim 2^s$.
\item All remaining caps go into $\mathcal{C}_0$; in other words, caps in $\mathcal{C}_0$ are such that $\# \theta'  \lesssim \lambda^{\frac{1}{2}} \delta^{\frac 3 2} +1$.
\end{itemize}

\begin{lem}
\label{lemmacounting2}
If $2^s \gg 1 + \lambda^{\frac{1}{2}} \delta^{\frac 32}$, then
$$
\# \mathcal{C}_s \lesssim \lambda \delta 2^{-2s}.
$$
\end{lem}

\begin{proof} This follows along the lines of the argument for Theorem 2.17 in~\cite{BourgainDemeter3}. Since $2^s \gg \lambda^{\frac 1 2} \delta^{\frac 32}$, which is the area of a cap, the set $\theta'$ is one-dimensional\footnote{Indeed, if a convex body $K \subset \mathbb{R}^2$ is symmetric with respect
to the origin, $a \in \mathbb{R}^2$ and $K \cap (a + \mathbb{Z}^2)$ has dimension 2, then $\# (a + \mathbb{Z}^2) \cap K \lesssim |K|$.}. It consists of colinear points, with equal spacing. We now split $\mathcal{S}_s$ into
$$
\mathcal{S}_{s,m}  = \{ \theta \in \mathcal{C},  \; \#  \theta'  \sim 2^s, \; \mbox{and two consecutive points in $\theta' $ are $\sim 2^m$ apart} \}. 
$$

On the one hand, we can associate to $\theta \in \mathcal{S}_{s,m}$ its direction $d_\theta$, which is the difference between two consecutive points in $\theta'$. This is a vector  of $\mathbb{Z}^2$ with magnitude $\sim 2^m$; therefore, the number of possible directions is $\lesssim 2^{2m}$.

On the other hand, the angle between $\theta'$ and the major axis of $\theta$ is $\lesssim \delta 2^{-m -s}$. If we consider a subcollection of $\theta \in \mathcal{S}_{s,m}$ with angular separation $\gg \delta 2^{-m -s}$, then the directions $d_\theta$ are distinct. Such an angular separation can be achieved by labeling all caps in the collection $\mathcal{C}$, starting at one cap, and following the circle; and then keeping only those caps which are in a fixed class modulo $\sim \lambda^{\frac 1 2} \delta^{\frac 12} 2^{-m-s}$.

These arguments show that
\begin{equation}
\label{boundSms}
\# \mathcal{S}_{m,s} \lesssim \lambda^{\frac 12} \delta^{\frac 12} 2^{-m-s} 2^{2m} = \lambda^{\frac 12} \delta^{\frac 12} 2^{m-s}.
\end{equation}
Summing over $2^m \lesssim \lambda^{\frac 12} \delta^{\frac 12} 2^{-s}$ gives the desired result.
\end{proof}

\begin{lem}
\label{lemmacounting1}
If $\lambda \delta > 1$,
$$
\# ( \mathcal{A}_{\lambda,\delta})' \lesssim \lambda^{1+\epsilon} \delta.
$$
\end{lem}
\begin{proof} 
There are $\sim \lambda \delta$ integers $n$ such that $\lambda - \delta < \sqrt{n} < \lambda + \delta$. For each such integer, there are at most $O(\lambda^\epsilon)$ solutions of $x^2 + y^2 =n$, by the divisor bound in $\mathbb{Z}[i]$.
\end{proof}

\subsection{$L^p$ estimates on caps with bounded numbers of points}

\label{sectionLp}

The basic idea behind the $L^4$ and $L^6$ estimates which are stated below is to break down our spectral projectors into projections on caps. In order to do so, 
 we choose first a smooth partition of unity $(\chi_\theta)$ associated to the collection $\mathcal{C}$:
$$
\operatorname{Supp} \chi_\theta \subset \theta, \qquad \mbox{and} \qquad \sum_{\theta \in \mathcal{C}}  \chi_\theta = 1 \quad \mbox{on $\mathcal{A}_{\lambda,\delta}$}.
$$
and define next the Fourier multipliers
$$
P^\theta = \chi^\theta(D).
$$

\begin{lem}[$L^4$ estimate by the bilinear argument]
\label{lemmaL4}
Let $f$ be a function on the torus whose Fourier support $S \subset \mathcal{A}_{\lambda,\delta}$ is such that for any $\theta \in \mathcal{C}$, $\# \left( S \cap \theta \right)' \leq N$.
Then
$$
\| f \|_{L^4} \lesssim_\epsilon N^{\frac 14} \| f \|_{L^2}.
$$
\end{lem}

\begin{proof} The argument essentially follows that of Proposition 2.4 in~\cite{BourgainBurqZworski}. Since $\widehat{f}$ is supported in $\mathcal{A}_{\lambda,\delta}$,
$$
\| P_{\lambda,\delta} f \|_{L^4} = \left\| \sum_{\theta,\widetilde{\theta} \in \mathcal{C}} P^\theta f \, P^{\widetilde{\theta}} f \right\|_{L^2}^{\frac 12}.
$$
The key geometrical observation is that, up to exchanging the roles of $\theta_1$ and $\theta_2$, the supports of $P^{\theta_1} f P^{\widetilde{\theta}_1} f$ and $P^{\theta_2} f P^{\widetilde{\theta}_2} f$ are disjoint unless 
$$
\operatorname{dist}(\theta_1,\theta_2) + \operatorname{dist}(\widetilde{\theta}_1, \widetilde{\theta}_2) \lesssim \lambda^{\frac 12} \delta^{\frac 12}.
$$ 
As a consequence, almost orthogonality followed by H\"older's inequality gives the bound
\begin{align*}
\| P_{\lambda,\delta} f \|_{L^4}  \lesssim \left( \sum_{\theta,\widetilde{\theta}} \left\| P^\theta f P^{\widetilde{\theta}} f \right\|_{L^2}^2 \right)^{\frac 14}
\lesssim \left( \sum_{\theta,\widetilde{\theta} }  \left\| P^\theta f \right\|_{L^4}^2 \left\| P^{\widetilde{\theta}} f \right\|_{L^4}^2 \right)^{\frac 14}
\end{align*}
Applying successively interpolation, the Cauchy-Schwarz inequality and Lemma~\ref{lemmacounting2}, we can estimate
$$
\| P^\theta f \|_{L^4} \lesssim \| P^\theta f \|_{L^\infty}^{\frac 12} \| P^\theta f \|_{L^2}^{\frac 12} \lesssim \left( \# \theta' \right)^{\frac 14} \| P^\theta f \|_{L^2} \leq N^{\frac 14} \| P^\theta f \|_{L^2}.
$$
Injecting this inequality in the previous estimate, and using once again almost orthogonality,
$$
\| P_{\lambda,\delta} f \|_{L^4} \lesssim N^{\frac 14} \left( \sum_{\theta,\widetilde{\theta}}  \left\| P^\theta f \right\|_{L^2}^2 \left\| P^{\widetilde{\theta}} f \right\|_{L^2}^2 \right)^{\frac 14} \lesssim N^{\frac 14} \| f \|_{L^2}.
$$
\end{proof}

\begin{lem}[The $L^6$ estimate by $\ell^2$ decoupling]
\label{lemmadecoupling}
Let $f$ be a function on the torus whose Fourier support $S \subset \mathcal{A}_{\lambda,\delta}$ is such that for any $\theta \in \mathcal{C}$, $\# \left( S \cap \theta \right)' \leq N$.
Then
$$
\| f \|_{L^6} \lesssim_\epsilon \lambda^\epsilon \delta^{-\epsilon} N^{\frac 13} \| f \|_{L^2}.
$$
\end{lem}

\begin{proof} The proof follows from ~\cite{GermainMyerson1}, the fundamental ingredient being the $\ell^2$ decoupling of Bourgain-Demeter~\cite{BourgainDemeter3}.
Writing $f$ as the sum of its Fourier series $f = \sum a_k e^{2\pi i k \cdot x}$ and changing variables to $X = \lambda x$ and $K = k/ \lambda$,
\begin{align*}
\| f \|_{L^6(\mathbb{T}^2)} & = \left\| \sum_{k \in \mathbb{Z}^2} a_k e^{2\pi i k \cdot x} \right\|_{L^6(\mathbb{T}^2)} 
 \lesssim \left( \frac{\delta}{\lambda} \right)^{\frac 13} \left\| \phi \left( \frac{\delta X}{\lambda} \right) \sum_{K \in \mathbb{Z}^2/\lambda}  a_{\lambda K} e^{2\pi i K \cdot X} \right\|_{L^{6}(\mathbb{R}^2)},
\end{align*}
where the cutoff function $\phi$ can be chosen to have compactly supported Fourier transform. As a result, the Fourier transform of the function on the right-hand side is supported on a $\delta/\lambda$-neighborhood of $\mathbb{S}^{1}$. It can be written as a sum of functions which are supported on caps $\theta/\lambda$ with dimension $\sim \frac{\delta}{\lambda} \times \frac{\delta^{\frac 12}}{\lambda^{\frac 12}}$ (recall from Section~\ref{sectioncounting} that the collection $\mathcal{C}$ of caps $\theta$ provides an almost disjoint covering of $\mathcal{A}_{\lambda,\delta}$) :
$$
\phi \left( \frac{\delta X}{\lambda} \right) \sum_{K \in \mathbb{Z}^2/\lambda}  a_{\lambda K} e^{2\pi i K \cdot X} = \sum_\theta \phi \left( \frac{\delta X}{\lambda} \right)  \sum_{K \in \mathbb{Z}^2 / \lambda} \chi_\theta(\lambda K) a_{\lambda K} e^{2\pi i K \cdot X}.
$$
By $\ell^2$ decoupling, the $L^6$ norm above is bounded by
$$
\lesssim_\epsilon \left( \frac{\delta}{\lambda} \right)^{\frac 13 - \epsilon} \left( \sum_{\theta} \left\|  \phi \left( \frac{\delta X}{\lambda} \right)  \sum_{K \in \mathbb{Z}^2 / \lambda} \chi_\theta(\lambda K) a_{\lambda K} e^{2\pi i K \cdot X}\right\|_{L^6(\mathbb{R}^2)}^2  \right)^{\frac 12}.
$$
At this point, we use the inequality
$$
\mbox{if $p \geq 2$}, \qquad \| g \|_{L^p(\mathbb{R}^2)} \lesssim \| g \|_{L^2(\mathbb{R}^2)} | \operatorname{Supp} \widehat{g} |^{\frac{1}{2} - \frac{1}{p}},
$$
which follows by applying successively the Hausdorff-Young and H\"older inequalities, and finally the Plancherel equality. We use this inequality for $g(X) =  \phi \left( \frac{\delta X}{\lambda} \right)  \sum_{K}  \chi_\theta(\lambda K) a_{\lambda K} e^{2\pi i K \cdot X}$. Since $\# (S \cap \theta)' \leq N$, its Fourier transform is supported on the union of at most $N$ balls of radius $O( \delta / \lambda) $, giving $| \operatorname{Supp} \widehat{f} | \lesssim N \delta^2 \lambda^{-2}$. Thus the $L^6$ norm we are trying to bound is less than
\begin{align*}
 & \lesssim  \left( \frac{\delta}{\lambda} \right)^{\frac 1 3 -\epsilon} (N \delta^2 \lambda^{-2})^{\frac 13} 
 \left( \sum_{\theta} \left\|  \phi \left( \frac{\delta X}{\lambda} \right)  \sum_{K \in \mathbb{Z}^2 / \lambda} \chi_\theta(\lambda K) a_{\lambda K} e^{2\pi i K \cdot X}\right\|_{L^2(\mathbb{R}^2)}^2  \right)^{\frac 12} 
 \end{align*}
 By almost orthogonality and periodicity of the Fourier series, this is in turn bounded by
\begin{align*}
& \lesssim \left( \frac{\delta}{\lambda} \right)^{1 - \epsilon} N^{\frac 13}  \left\|  \phi \left( \frac{\delta X}{\lambda} \right) \sum_{K \in \mathbb{Z}^2/\lambda} a_{\lambda K} e^{2\pi i K \cdot X} \right\|_{L^2(\mathbb{R}^2)} \lesssim \left( \frac{\delta}{\lambda} \right)^{- \epsilon} N^{\frac 13} \| f \|_{L^2(\mathbb{T}^2)},
\end{align*}
which is the desired estimate.
\end{proof}

\subsection{Interpolation}

Interpolating between the estimates proved in the previous subsections enables us to prove the following theorem.

\label{sectionreal}

\begin{thm} \label{thm1}
\begin{itemize}
\item[(i)] If $p \in [2, 6)$ and $\lambda^{-1+\kappa} < \delta < \lambda^{-\kappa}$ for some $\kappa>0$,
$$
\| P_{\lambda,\delta} \|_{L^2 \to L^p} \lesssim_{p,\kappa} (\lambda \delta)^{\frac{1}{4} - \frac{1}{2p}}.
$$
\item[(ii)] If either $p \in [6, 10]$ and $\delta > \lambda^{-1}$, or $p \in [10,\infty]$ and $\delta > \lambda^{- \frac 13}$,
$$
\| P_{\lambda,\delta} \|_{L^2 \to L^p} \lesssim_\epsilon \lambda^\epsilon\left[ \lambda^{\frac{1}{2} - \frac{2}{p}} \delta^{\frac 12} + (\lambda \delta)^{\frac{1}{4} - \frac{1}{2p}} \right].
$$
\end{itemize}
\end{thm}

\begin{proof} \underline{Decomposition of $P_{\lambda,\delta}$}
Recall that the Fourier multipliers $P^\theta$ were defined in the previous section. They are now used to set
$$
P^s = \sum_{\theta \in \mathcal{C}_s} P^\theta \qquad \mbox{and} \qquad P^0 = \sum_{\theta \in \mathcal{C}_0} P^\theta
$$
as well as
$$
P^s_{\lambda,\delta} = P_{\lambda,\delta} P^s \qquad \mbox{and} \qquad P^0_{\lambda,\delta} = P_{\lambda,\delta} P^0.
$$
We will split $P_{\lambda,\delta}$ into
$$
P_{\lambda,\delta} = P^0_{\lambda,\delta} + \sum_s P^s_{\lambda,\delta},
$$
and estimate the different summands on the right-hand side by interpolating between $L^2 \to L^p$ bounds, with $p=4,6,\infty$.

\medskip
\noindent \underline{Basic bounds}
We learn from Lemma~\ref{lemmaL4} and Lemma~\ref{lemmadecoupling} that, if $\delta \gtrsim \lambda^{-1}$ and $2^s \gg 1 + \lambda^{\frac{1}{2}} \delta^{\frac{3}{2}}$,
\begin{align*}
& \| P_{\lambda,\delta}^0 \|_{L^2 \to L^4} \lesssim  \lambda^{\frac 18} \delta^{\frac 38} + 1 \\
& \| P_{\lambda,\delta}^s \|_{L^2 \to L^4} \lesssim 2^{\frac s4}  \\
& \| P_{\lambda,\delta}^0 \|_{L^2 \to L^6} \lesssim_\epsilon \lambda^\epsilon \left[ \lambda^{\frac 16} \delta^{\frac 12} + 1 \right] \\
& \| P^s_{\lambda,\delta} f \|_{L^2 \to L^6} \lesssim_\epsilon \lambda^\epsilon 2^{\frac{s}{3}}.
\end{align*}
Furthermore, lemmas~\ref{lemmacounting2} and~\ref{lemmacounting1} give the bounds
\begin{align*}
& \| P_{\lambda,\delta}^0 \|_{L^2 \to L^\infty} \lesssim \lambda^\epsilon (\lambda \delta)^{\frac{1}{2}} \\ 
& \| P_{\lambda,\delta}^s \|_{L^2 \to L^\infty} \lesssim \lambda^{\frac 12} \delta^{ \frac 12} 2^{- \frac s2}  \\ 
\end{align*}

\medskip
\noindent \underline{The case $2 \leq p \leq 4$} By the basic bounds above,
$$
\| P_{\lambda,\delta} \|_{L^2 \to L^4} \lesssim \| P_{\lambda,\delta}^0 \|_{L^2 \to L^4} + \sum_j \| P_{\lambda,\delta}^s \|_{L^2 \to L^4} \lesssim  \lambda^{\frac 18} \delta^{\frac 38} + 1 + \sum_{1 \leq 2^s \leq \lambda^{\frac 12} \delta^{\frac 12}}  2^{\frac s 4} \lesssim (\lambda \delta)^{\frac 1 8}.
$$
This is the desired bound if $p=4$, and the case $2 \leq p \leq 4$ follows from interpolation with the trivial case $p=2$.

\medskip
\noindent \underline{Bounding $P^0_{\lambda,\delta}$}
Interpolating between the basic bounds for $P^0_{\lambda,\delta}$ bounds gives, if $4 \leq p \leq 6$,
\begin{align*}
& \| P^0_{\lambda,\delta} \|_{L^2 \to L^p} \lesssim \lambda^\epsilon \left[ 1 + \lambda^{\frac 14 - \frac 1 {2p}} \delta^{\frac 34 - \frac 3 {2p}} \right],
\end{align*}
which is consistent with the conjecture if $\lambda^{-1+ \kappa} < \delta < \lambda^\kappa$.

If $6 \leq p \leq \infty$, we obtain instead
\begin{equation*}
 \| P^0_{\lambda,\delta} \|_{L^2 \to L^p} \lesssim \lambda^\epsilon \left[ (\lambda \delta)^{\frac 12 - \frac 3p} + \lambda^{\frac 12 - \frac 2p} \delta^{\frac 12} \right].
\end{equation*}
 This is consistent with the conjecture with $\epsilon$ loss if $(\lambda \delta)^{\frac 12 - \frac 3p} + \lambda^{\frac 12 - \frac 2p} \delta^{\frac 12} \lesssim \lambda^{\frac 1 2 - \frac 2p} \delta^{\frac 12} + (\lambda \delta)^{\frac 14 - \frac 1 2p}$, which is the case if $\delta > \lambda^{-\frac 13}$ or $p \leq 10$.

\medskip
\noindent \underline{Bounding $\sum_s P^s_{\lambda,\delta}$ if $4 \leq p \leq 6$.} Interpolating between the basic bounds for $P^s_{\lambda,\delta}$ from $L^2$ to $L^4$ and $L^2$ to $L^\infty$,\begin{align*}
\| P^s_{\lambda,\delta} \|_{L^2 \to L^p} \lesssim \lambda^{\frac 12 - \frac 2 {p}} \delta^{\frac 12 - \frac 2 {p}} 2^{s \left( -\frac 12 + \frac 3p \right)} \qquad \mbox{if $2^s \gg 1 + \lambda^{\frac 12} \delta^{\frac 32}$} .
\end{align*}
Therefore, if $p<6$,
$$
\sum_{1 + \lambda^{\frac 12} \delta^{\frac 32} \ll 2^s \lesssim \lambda^{\frac 12} \delta^{ \frac 1 2}} \| P^s_{\lambda,\delta} \|_{L^2 \to L^p} \lesssim  \lambda^{\frac 14 - \frac 1 {2p}} \delta^{\frac 14 - \frac 1 {2p}},
$$
which is consistent with the conjecture.

\medskip
\noindent \underline{Bounding $\sum_s P^s_{\lambda,\delta}$ if $p \geq 6$.} Interpolating between the basic bounds for $P^s_{\lambda,\delta}$ from $L^2$ to $L^6$ and $L^2$ to $L^\infty$,
\begin{equation*}
\| P^s_{\lambda,\delta} \|_{L^2 \to L^p} \lesssim_\epsilon \lambda^\epsilon \lambda^{\frac 12 - \frac 3 {p}} \delta^{\frac 12 - \frac 3 {p}} 2^{s \left( -\frac 12 + \frac 5p \right)} \qquad \mbox{if $2^s \gg 1 + \lambda^{\frac 12} \delta^{\frac 32}$} .
\end{equation*}
Summing over $2^j$ gives if $p \leq 10$
\begin{align*}
\sum_{1 + \lambda^{\frac 12} \delta^{\frac 32} \ll 2^s \lesssim \lambda^{\frac 12} \delta^{\frac 12}} \| P^j_{\lambda,\delta} \|_{L^2 \to L^p} \lesssim \lambda^\epsilon \lambda^{\frac 12 - \frac 3 {p}} \delta^{\frac 12 - \frac 3 {p}}  \sum_{ 2^s \lesssim \lambda^{\frac 12}\delta^{\frac 1 2}} 2^{s \left( -\frac 12 + \frac 5p \right)} \sim \lambda^\epsilon (\lambda \delta)^{\frac{1}{4} - \frac{1}{2p}},
\end{align*}
which is consistent with the conjecture with $\epsilon$ loss. 

If we assume now that $p \geq 10$,
\begin{align*}
\sum_{1 + \lambda^{\frac 12} \delta^{\frac 32} \ll 2^s \lesssim \lambda^{\frac 12} \delta^{\frac 12}} \| P^j_{\lambda,\delta} \|_{L^2 \to L^p} & \lesssim  \lambda^\epsilon \lambda^{\frac 12 - \frac 3 {p}} \delta^{\frac 12 - \frac 3 {p}} \sum_{2^s \gg 1 + \lambda^{\frac 12} \delta^{\frac 32}}  2^{s \left( -\frac 12 + \frac 5p \right)}\\
& \sim \lambda^\epsilon \left[ \lambda^{\frac 14 - \frac 1 {2p}} \delta^{- \frac 14 + \frac 9 {2p}} + (\lambda \delta)^{\frac{1}{4} - \frac 1{2p}} \right]
\end{align*}
This is $\lesssim \lambda^\epsilon \left[ \lambda^{\frac 1 2 - \frac 2 p} \delta^{\frac 12} + (\lambda \delta)^{\frac 14 - \frac 1 {2p}} \right]$ if and only if $\delta > \lambda^{-\frac{1}{3}}$.
\end{proof}

\section{Dyadic decomposition of the kernel}

In this section, we will prove the following theorem, which validates Conjecture~\ref{conjproj} for some range of $(p,\lambda,\delta)$. The idea of the proof is to decompose dyadically (in the Poisson summation formula) the kernel of the spectral projector.

\label{sectiondyadic}

\begin{thm} \label{thm2} For $p \geq 6$ and $\epsilon>0$,
$$
\| P_{\lambda,\delta} \|_{L^2 \to L^p} \lesssim_\epsilon \lambda^{\frac{1}{2} - \frac{2}{p}} \delta^{\frac 12} \qquad \mbox{if} \qquad \delta > \min \left( \lambda^{- \frac{1 - \frac{6}{p}}{3 - \frac 2p}} , \lambda^{-\frac{10 - \frac{64}{p}}{29 - \frac{14}{p}}+\epsilon} \right).
$$
\end{thm}

\begin{rem}
The proof given below uses a pointwise bound on a two-dimensional exponential sum, which is borrowed from M\"uller~\cite{Muller} and enables us to prove the theorem for scales $\delta$ (moderately) smaller than $\lambda^{-\frac 1 3}$. The bound in~\cite{Muller} has the advantage of being robust and admitting a rather simple proof, by the Van der Corput method, while providing an improvement over the trivial estimate. But it is certainly not optimal; in particular, the methods recounted in Huxley~\cite{Huxley} leading to Huxley~\cite{Huxley2} could give further improvements, though they do not seem immediately applicable to the sum under under consideration.
\end{rem}

\begin{proof}
\noindent \underline{Decomposition of the kernel}
For technical reasons, it will be more convenient to consider a mollified version of the projector $P_{\lambda,\delta}$: let 
$$
P_{\lambda,\delta}^\flat e^{2\pi ik \cdot x} = \chi_{\lambda,\delta} (k)e^{2\pi ik \cdot x} ,
$$
where the function $\chi_{\lambda,\delta}$ is a 
nonnegative and smooth cutoff function adapted to the annulus $\mathcal{A}_{\lambda,\delta}$, namely $\chi_{\lambda,\delta}(x) >c>0$ if $x \in \mathcal{A}_{\lambda,\delta}$. 
To be more specific, it is defined as follows: consider the superficial measure $d\sigma_\lambda$ induced by the Lebesgue measure on the circle of center 0 and radius $\lambda$ (in other words, $d\sigma_\lambda$ is the uniform measure on that circle with total mass $2\pi \lambda$). Consider furthermore a positive function $\chi$ whose Fourier transform is compactly supported. Finally, let
$$
\chi_{\lambda,\delta} = \delta^{-1} \chi (\delta^{-1} \cdot) * d\sigma_\lambda.
$$

We now introduce a dyadic partition of unity, namely $\mathcal{C}_0^\infty$, nonnegative functions $\varphi$ and $\psi$ such that
$$
\varphi(x) + \sum_{M \in 2^\mathbb{N}} \psi \left( \frac{x}{M} \right) =1,
$$
and furthermore $\varphi$ is supported in a ball, and $\psi$ in an annulus.

Denoting $\Phi_{\lambda,\delta}^\flat$ the kernel of $P_{\lambda,\delta}^\flat$, it can be written by Poisson summation as
$$
\Phi_{\lambda,\delta}^\flat(x) = \sum_{k \in \mathbb{Z}^2} \chi_{\lambda,\delta}(k) e^{2\pi i k \cdot x} = \sum_{m \in \mathbb{Z}^2} \widehat{\chi_{\lambda,\delta}} (m - x),
$$
(where $\widehat{\cdot}$ stands for the Fourier transform on $\mathbb{R}$) 
and further decomposed, with the help of the partition of unity, into
\begin{align*}
\Phi_{\lambda,\delta}^\flat(x) & = \sum_{m \in \mathbb{Z}^2} \varphi(m - x) \widehat{\chi_{\lambda,\delta}} (m - x) + \sum_{m \in \mathbb{Z}^2} \sum_{M \in 2^\mathbb{N}} \psi \left( \frac{m-x}{M} \right) \widehat{\chi_{\lambda,\delta}} (m - x) \\
& = \Phi_{\lambda,\delta}^{\flat,0}(x) + \sum_{M \in 2^\mathbb{N}} \Phi_{\lambda,\delta}^{\flat,M}(x).
\end{align*}

Finally, the operators associated to these convolution kernels are denoted by $P_{\lambda,\delta}^{\flat,0}$ and $P_{\lambda,\delta}^{\flat,M}$.

\bigskip

\noindent \underline{Asymptotics of $\widehat{\chi_{\lambda,\delta}}$.} Denoting $J(\xi)$ for the Fourier transform of $d\sigma_1$ (superficial measure on the unit sphere), it follows from the definition of $\chi_{\lambda,\delta}$ that
$$
\widehat{\chi_{\lambda,\delta}}(\xi) = \lambda \delta J(\lambda \xi) \widehat{\chi}(\delta \xi).
$$
The function $J$ is smooth, and its asymptotic expansion is well-known
\begin{equation}
\label{asymptoticsJ}
J(\xi) \sim \frac{e^{i|\xi|}}{|\xi|^{\frac 12}}  \sum_{j=0}^\infty a_j |\xi|^{-j} + \frac{e^{-i|\xi|}}{|\xi|^{\frac 12}}  \sum_{j=0}^\infty b_j |\xi|^{-j}.
\end{equation}
We refer to~\cite{Stein}, Chapter VIII, for the proof of this statement and the meaning of the series in the equivalent (see in particular Proposition 3). 

\medskip

\noindent \underline{Bounding $P_{\lambda,\delta}^{\flat,0}$.} By Young's inequality and the above expansion,
$$
\| P_{\lambda,\delta}^{\flat,0} \|_{L^{p'} \to L^p} \lesssim \| \Phi_{\lambda,\delta}^{\flat,0} \|_{L^{\frac p 2}} \lesssim \| \widehat{\chi_{\lambda,\delta}} \|_{L^{\frac p 2}} \lesssim \lambda \delta \| J(\lambda \xi) \|_{L^{\frac p 2}} \lesssim \lambda^{1- \frac 4p} \delta \qquad \mbox{if $p>8$}.
$$
To treat the case $p \in [6,8]$, we can invoke the fact that the operator on $\mathbb{R}^2$ with symbol $\chi_{\lambda,\delta}(\xi)$ has $L^{p'} \to L^p$ operator norm $\lesssim \lambda^{1 - \frac{4}{p}} \delta$, which follows from the Stein-Tomas theorem~\cite{Tomas,Stein}. As a consequence, the operator with convolution kernel $\varphi \widehat{\chi_{\lambda,\delta}}$ has operator norm $\lesssim \lambda^{1 - \frac{4}{p}} \delta$ (indeed, the function $\varphi \widehat{\chi_{\lambda,\delta}}$ can be written under the form $\delta \widehat{\chi_{\lambda,1}}$ by modifying the cutoff function).

Since $\varphi$ is compactly supported, this implies the desired bound for the operator $P_{\lambda,\delta}^{\flat,0}$, whose convolution kernel is given by the periodization of $\varphi \widehat{\chi_{\lambda,\delta}}$.

\medskip

\noindent \underline{Bounding $P_{\lambda,\delta}^{\flat,M}$.} This will be achieved by interpolating between $L^2 \to L^2$ and $L^1 \to L^\infty$ bounds, in a manner which is reminiscent of the classical proof of the Stein-Tomas theorem~\cite{Tomas}.
Before doing so, we observe that the range of $M$ can be restricted to $M \lesssim \delta^{-1}$; this follows from the fact that $\widehat{\chi}$ is compactly supported, and the formula for $\widehat{\chi_{\lambda,\delta}}$ above.

\begin{itemize}
    \item To obtain the $L^2 \to L^2$ bound, we deduce from the definition of the kernel $\Phi_{\lambda,\delta}^{\flat,M}$ and Poisson summation that $P_{\lambda,\delta}^{\flat,M}$ is the Fourier multiplier on the torus with symbol $M^2 \psi(M\cdot) * \chi_{\lambda,\delta}$. Therefore,
    $$
\| P_{\lambda,\delta}^{\flat,M} \|_{L^2 \to L^2} \lesssim \| M^2 \psi(M\cdot) * \chi_{\lambda,\delta} \|_{L^\infty} \lesssim M \delta \qquad \mbox{if $M \lesssim \delta^{-1}$}
$$
\item To obtain the $L^1 \to L^\infty$ bound, we rely on exponential sum estimates. Reducing the asymptotics of $\widehat{\chi_{\lambda,\delta}}$ to its leading order (lower order terms being easier to treat), we need to bound
$$
S_{\lambda,M,x} = \lambda^{\frac 12} \delta \sum_{\substack{n \in \mathbb{Z}^2}} \psi\left( \frac {n-x} M \right) \frac{e^{i \lambda |n-x|}}{\langle n - x \rangle^{\frac 12}}.
$$
A first obvious bound is
$$
| S_{\lambda,M,x} | \lesssim \lambda^{\frac 12} \delta M^{\frac 32}.
$$
Furthermore, this sum can be written under the form $\lambda^{\frac 12} \delta M^{-\frac 12} \sum_{\substack{ |m| \sim M}} W(m) e^{i f(m)}$, which is the form considered in~\cite{Muller} (see also~\cite{Guo}). In order to apply Theorem 2 in~\cite{Muller}, we note that $W$ and $f$ satisfy the required derivative bounds; as for the condition on the determinant of iterated derivatives, it is verified thanks to Lemma 3 in that same paper. This yields the bound
$$
| S_{\lambda,M,x} | \lesssim_\epsilon \lambda^{\frac 12 + \omega} \delta M^{\frac{3}{2}-(q+1)\omega + \epsilon} \quad \mbox{if $\lambda \geq M^{q - 2 + \frac{2}{Q}}$}, \quad \mbox{with} \quad \omega = \frac{2}{4(Q-1)+2Q}, \quad Q =2^q.
$$
Choosing $q = 3$, this means that
\begin{equation*}
| S_{\lambda,M} | \lesssim_\epsilon \lambda^{\frac {6} {11}} \delta M^{\frac{29}{22} + \epsilon} \qquad \mbox{if $\lambda > M^{\frac 54}$}.
\end{equation*}
Overall, still under the assumption that $\lambda > M^{\frac 54}$,
$$
\| P_{\lambda,\delta}^{\flat,M} \|_{L^1 \to L^\infty} \lesssim \| \Phi_{\lambda,\delta}^{\flat,M} \|_{L^\infty} \lesssim \min(\lambda^{\frac 12} \delta M^{\frac 32} ,  \lambda^{\frac {6} {11}} \delta M^{\frac{29}{22} + \epsilon})
$$
\end{itemize}
Interpolating between these two bounds, we find that
$$
\| P_{\lambda,\delta}^{\flat,M} \|_{L^{p'} \to L^p} \lesssim_\epsilon \min ( M^{\frac 32 - \frac 1p} \lambda^{\frac 12 - \frac 1p} \delta, M^{\frac {29} {22} - \frac 7 {11p} + \epsilon} \lambda^{\frac 6 {11} - \frac{12}{11p}} \delta).
$$
Since $M \lesssim \delta^{-1}$, this is $< \lambda^{1 - \frac{4}{p}} \delta$ provided
$$
\delta > \min \left( \lambda^{- \frac{1 - \frac{6}{p}}{3 - \frac 2p}} , \lambda^{-\frac{10 - \frac{64}{p}}{29 - \frac{14}{p}}+\epsilon} \right).
$$

\medskip
\noindent \underline{Conclusion of the argument} Reconstructing $P^{\flat}_{\lambda,\delta}$ as the sum of $P^{\flat,M}_{\lambda,\delta}$ for $0 \leq M \lesssim \delta^{-1}$ gives the $L^{p'} \to L^p$ bound $\lambda^{1 - \frac{4}{p}} \delta$. By the classical $TT^*$ argument, this implies a $L^2 \to L^p$ bound $\lambda^{\frac 12 - \frac{2}{p}} \delta^{\frac 12}$ for the operator with symbol $\sqrt{\chi_{\lambda,\delta}}$, from which we deduce the desired $L^2 \to L^p$ bound for $P_{\lambda,\delta}$.\end{proof}

\section{The conjecture on the $L^p$ norm of the convolution kernel}

This section is dedicated to Conjecture~\ref{conjkernel}, which was introduced in the introduction. We show how it is related to Conjecture~\ref{conjproj} on the one hand, to additive energies on the other hand, and then use small cap decoupling to prove it with an $\epsilon$ loss,  if $p \leq 6$ or $\delta > \lambda^{-\frac 13}$.

\label{sectionconvolution}

\subsection{Partial equivalence of Conjecture~\ref{conjproj} and Conjecture~\ref{conjkernel}}

\begin{lem} \label{equivalence}
\begin{itemize}
\item[(i)] If $\delta > \lambda^{-1 + \frac{8}{p+2}}$, Conjecture~\ref{conjproj} for $(\delta,\lambda,p)$  implies Conjecture~\ref{conjkernel} for $(\delta,\lambda,p)$.
\item[(ii)] If $p \geq 4$ and $\delta > \lambda^{-1 + \frac{4}{p}}$, Conjecture~\ref{conjkernel} for $(\delta,\lambda,p)$  implies Conjecture~\ref{conjproj} for $(\delta,\lambda,2p)$.
\end{itemize}
\end{lem}

\begin{proof} Let us assume first that Conjecture~\ref{conjproj} holds for $(\delta,\lambda,p)$ and $\delta > \lambda^{-1 + \frac{8}{p+2}}$. Note first that
$$
\|  \Phi_{\lambda,\delta} \|_{L^2} = \| \Phi_{\lambda,\delta} \|_{L^\infty}^{\frac 12} = \| P_{\lambda,\delta} \|_{L^2 \to L^\infty} = (\lambda \delta)^{\frac 12},
$$
where we used Conjecture~\ref{conjproj} for $(\delta,\lambda,\infty)$, which is a consequence of the conjecture for $(\delta,\lambda,p)$.
Then, estimating $\|  \Phi_{\lambda,\delta} \|_{L^2}$ by Parseval's equality,
$$
\| \Phi_{\lambda,\delta} \|_{L^p} = \| P_{\lambda,\delta} \Phi_{\lambda,\delta} \|_{L^p} \leq \| P_{\lambda,\delta} \|_{L^2 \to L^p} \|  \Phi_{\lambda,\delta} \|_{L^2} \lesssim [ \lambda^{\frac{1}{2} - \frac 2p} \delta^{\frac{1}{2}} ] [ \lambda^{\frac 12} \delta^{\frac 12} ] = \lambda^{1 - \frac{2}{p}} \delta,
$$
which proves Conjecture~\ref{conjkernel}.

Conversely, let us assume that Conjecture~\ref{conjkernel} holds for $(\delta,\lambda,p)$ and $\delta > \lambda^{-1 + \frac{4}{p}}$. Then, by Young's inequality,
$$
\| P_{\lambda,\delta} f \|_{L^{2p}} = \| \Phi_{\lambda,\delta} * f \|_{L^{2p}} \lesssim \| \Phi_{\lambda,\delta} \|_{L^p} \| f \|_{L^{(2p)'}} \lesssim  \lambda^{1- \frac 2p} \delta \| f \|_{L^{(2p)'}}.
$$
This means that
$$
\| P_{\lambda,\delta} \|_{L^{(2p)'} \to L^{2p}} \lesssim \lambda^{1- \frac 2p} \delta.
$$
By the classical $TT^*$ argument,
$$
\| P_{\lambda,\delta} \|_{L^{2} \to L^{2p}}= \| P_{\lambda,\delta} \|_{L^{{(2p)'}} \to L^{2p}}^{\frac 12} \lesssim \lambda^{\frac 12- \frac 1p} \delta^{\frac 12},
$$
which proves Conjecture~\ref{conjproj} for $(\delta,\lambda,2p)$.
\end{proof}

\subsection{Link to additive energies} \label{additivelink}
For $p$ an even number, the $\frac{p}{2}$-additive energy of the set $\Lambda \subset \mathbb{R}^2$ is defined as
$$
\mathbb{E}_{\frac p2}(\Lambda) = \# \{ (x_1,\dots,x_{p}) \in \Lambda^{p} \; \mbox{such that} \: x_1 + \dots + x_{\frac p2} = x_{\frac{p}{2}+1} + \dots + x_{p} \}.
$$
If $p$ is an even number,
$$
\mathbb{E}_{\frac p 2}(\mathcal{A}_{\lambda,\delta}') = \| \Phi_{\lambda,\delta} \|_{L^p(\mathbb{T}^2)}^p,
$$
so that the conjecture can be reformulated, for even $p$, as
$$
\mathbb{E}_{\frac p2}(\mathcal{A}'_{\lambda,\delta}) \sim \lambda^{p-2} \delta^p + (\lambda \delta)^{\frac p2}.
$$

How can this conjecture be interpreted in terms of additive energies? The second term on the above right-hand side comes from diagonal contributions, in other words from the universal inequality
$$
\mathbb{E}_{\frac p2}(\mathcal{A}'_{\lambda,\delta}) \geq | \mathcal{A}'_{\lambda,\delta}|^{\frac{p}{2}}
$$
For the first term on the right-hand side, note that there are $\sim(\lambda\delta)^{p/2}$ possible sums of $p/2$ elements of $\mathcal{A}'_{\lambda,\delta}$, and they all lie inside the ball of radius
$\sim \lambda$. Moreover, two distinct sums will have to be at least 1 apart from each other. As a result of this, many pairs of sums are forced to coincide, and a simple application of Cauchy-Schwarz proves the first lower bound.   

Additive energies of annular sets $\mathcal{A}'_{\lambda,\delta}$ in dimension 2 were considered in~\cite{KrishnapurKurlbergWigman}, and then by Bombieri-Bourgain~\cite{BombieriBourgain} who could prove that
$$
\mathbb{E}_3(\mathcal{A}_{\lambda,0}' ) \lesssim [\# \mathcal{A}_{\lambda,0}' ]^{\frac 72}.
$$
They conjectured that the exponent $\frac{7}{2}$ can be replaced by $3$ (which is equivalent to the case $p=6$ of~\eqref{macareux}). Bourgain-Demeter~\cite{BourgainDemeter3} proved the conjecture for $p=6$, $\delta = \lambda^{-\frac 13}$ with an $\epsilon$-loss. Finally, the question was also considered in dimensions four and five, see~\cite{BourgainDemeter2}.

\subsection{Small cap decoupling}

\label{sectionprogress}

In this section we write $A\les B$ if $A\lesssim_\epsilon P^\epsilon B$ holds for all $\epsilon>0$, where $P$ is the scale parameter, typically denoted by $R$ or $\lambda$.

The chief tool for proving  Conjecture B in the range $4\le p\le 6$ is small cap decoupling, a result first proved in \cite{DGW}, and further refined in \cite{FGM}.

Let $I$ be a compact interval and let $\Gamma:I\to\R$ be a smooth curve with
\begin{equation}
\label{dd14}
\min_{\xi\in I}|\Gamma''(\xi)|>0.
\end{equation}
Let $0<\beta\le 1$. For $R\ge 1$, partition its $1/R$-neighborhood $\mathcal{N}_{1/R}(\Gamma)$ into tubular regions $\Gamma$ with length $R^{-\beta}$ and width/height $\sim 1/R$. We will call such $\gamma$ an $(R^{-\beta},R^{-1})$-cap. There are $\sim R^{\beta}$ such caps.

We introduce the Fourier projection onto $L^2(\gamma)$
$$f_\gamma(x)=\int_{\gamma}\widehat{f}(\xi)e(x\cdot \xi)d\xi.$$
Given a ball $B_R=B(x_0,R)$, we define the weight
$$w_{B_R}(x)=(1+\frac{|x-x_0|}{R})^{-100}.$$

The following was proved in \cite{DGW} for the parabola. The extension to arbitrary curves $\Gamma$ as above is fairly standard, see the sketch of proof below.

\begin{thm}[$l^p(L^p)$ small cap decoupling, \cite{DGW}]
Assume $f:\R^2\to\C$ has Fourier transform supported in $\mathcal{N}_{1/R}(\Gamma)$. Then for $4\le p\le \min(2+\frac2{\beta},6)$ and each ball $B_R$ we have
\begin{equation}
\label{dd1}
\|f\|_{L^p(B_R)}\les R^{\beta(\frac12-\frac1p)}(\sum_{\gamma}\|f_\gamma\|_{L^p(w_{B_R})}^p)^{1/p}.
\end{equation}
\end{thm}
When $\beta>1/2$ this is a good substitute for $l^2(L^p)$ decoupling
$$\|f\|_{L^p(B_R)}\les (\sum_{\gamma}\|f_\gamma\|_{L^p(w_{B_R})}^2)^{1/2},$$
which is only true if $\beta\le 1/2$. The reason for the failure of this inequality when $\beta>1/2$ is the fact that there are many (more precisely $R^{\beta-\frac12}$) consecutive caps $\beta$ inside
an essentially rectangular/flat cap $\tau$ with dimensions $(R^{-1/2},R^{-1})$.	

Parts of our forthcoming argument need a slightly stronger (via H\"older's inequality) version of \eqref{dd1}. This is a particular case of  Corollary 5 in \cite{FGM}.
\begin{thm}[$l^q(L^p)$ small cap decoupling, \cite{FGM}]
	Assume $f:\R^2\to\C$ has Fourier transform supported in $\mathcal{N}_{1/R}(\Gamma)$. Then for $4\le p\le \min(2+\frac2{\beta},6)$, $\frac1q=1-\frac{3}{p}$ and  each ball $B_R$ we have
	\begin{equation}
	\label{dd2}
	\|f\|_{L^p(B_R)}\les R^{\beta(\frac12-\frac1q)}(\sum_{\gamma}\|f_\gamma\|_{L^p(w_{B_R})}^q)^{1/q}.
	\end{equation}
\end{thm}
The  upper bound $p\le \min(2+\frac{2}{\beta},6)$ is sharp in both theorems. The value $2+\frac{2}{\beta}$ is called the critical exponent for small cap decoupling.

\begin{proof}[Sketch of proof:]
Let us comment on the extension of these theorems to general $C^2$ curves $\Gamma:[-1/2,1/2]\to\R$ satisfying $|\Gamma(x)|, |\Gamma'(x)|\lesssim 1$ and  $|\Gamma''(x)|\sim 1$ for $x\in[-1/2,1/2]$. 

The main step in the proof for  the parabola was proving a bilinear version of the small cap decoupling inequality. More precisely, the function $f$ is replaced with the geometric average $|f_1f_2|^{1/2}$ where the spectra of $f_1,f_2$ lie in $\sim 1$ separated parts of $\mathcal{N}_{1/R}$. The only relevance of this separation is that normals at points lying in the two pieces point in separated directions.

Two special tools were used to prove this bilinear inequality. One of them is Cordoba's inequality, whose validity and proof remain the same for  curves $\Gamma$ as above. The other one is a refined Kakeya inequality, which takes the same form for all $\Gamma$ with nonzero curvature. Indeed, curvature forces the spatial rectangles localizing wave packets to point in distinct directions. 

The remaining step in the proof for the parabola was a Whitney-type decomposition for this curve into smaller pieces, and the application of the previously mentioned bilinear small cap decoupling to each piece, via parabolic rescaling. The latter amounts to mapping a small arc on the parabola to the full scale-one parabola via an affine transformation (it is crucial that  affine transformations  commute with the Fourier transform). Strictly speaking, this strong form of parabolic rescaling fails for arbitrary curves. However, the following totally satisfactory analogue is true:  given any interval $J\subset [-1/2,1/2]$ of length $\Delta$ and centered at $c$, the affine map
$$(\xi,\eta)\to(\frac{\xi-c}{\Delta},\frac{\eta-\Gamma(c)-\Gamma'(c)\xi}{\Delta^2})$$
maps the arc $\Gamma:J\to\R$ to some $\Gamma':[-1/2,1/2]\to\R$ that has the same properties as $\Gamma$. And since the bilinear decoupling holds true with uniform bounds for such curves, the argument closes in the same way as in the case of the parabola. 
\end{proof}

In our applications of \eqref{dd1} and \eqref{dd2}, $\widehat{f}$ will be supported  only on a small number $N_{active}$ of the total number $N_{total}\sim R^{\beta}$ of caps $\gamma$. Then \eqref{dd1} gives
\begin{equation}
\label{dd3}
\|f\|_{L^p(B_R)}^p\les N_{total}^{\frac{p}2-1}N_{active}\max_{\gamma}\|f_\gamma\|_{L^p(w_{B_R})}^p,
\end{equation}
while \eqref{dd2} leads to the more favorable
\begin{equation}
\label{dd4}
\|f\|_{L^p(B_R)}^p\les N_{total}^{3-\frac{p}2}N_{active}^{p-3}\max_{\gamma}\|f_\gamma\|_{L^p(w_{B_R})}^p.
\end{equation}
When $\beta\le 1/2$, we have an even stronger estimate, as a consequence of $l^2(L^p)$ decoupling
\begin{equation}
\label{dd11}
\|f\|_{L^p(B_R)}^p\les N_{active}^{p/2}\max_{\gamma}\|f_\gamma\|_{L^p(w_{B_R})}^p.
\end{equation}
We will work with $\Gamma:[-1/2,1/2]\to\R$ given by $\Gamma(\xi)=\sqrt{1-\xi^2}$. In fact, for reasons of symmetry, we may as well work with the full circle $\mathbb{S}^1$. We rescale \eqref{dd3} and \eqref{dd4} to allow $\widehat{f}$ to be supported on $\mathcal{A}_{\lambda,\delta}$, for some $\delta\le 1$. If $\gamma$ are now $(\lambda(\frac{\delta}{\lambda})^{\beta},\delta)$-caps partitioning $\mathcal{N}_\delta(\lambda S^1) = \mathcal{A}_{\lambda,\delta}$, we find that if $\widehat{f}$ is supported on $ \mathcal{A}_{\lambda,\delta}$ then \eqref{dd3}, \eqref{dd4} and \eqref{dd11} hold with $R\sim 1/\delta$.
\medskip

\subsection{Some progress on the kernel conjecture}

The main result we prove is as follows. It is important to note that the difficult range $4<p<6$ does not follow by interpolating the easier endpoint cases $p=4,6$.

\begin{thm}[Square root cancellation at the critical exponent]
\label{dd12}	
Assume $p\in[4,6]$ and $\delta=\lambda^{\frac{4}{p}-1}$. Then
$$\|\Phi_{\lambda,\delta}\|_{L^p([0,1]^2)}\les (\lambda\delta)^{1/2}.$$	
\end{thm}
\begin{proof}
Note first that
\begin{equation}
\label{dd7}
\delta\ge \lambda^{-1/3},
\end{equation}
so that in particular $\delta^{-1} < \sqrt{\lambda \delta}$. We cover $\mathcal{A}_{\lambda,\delta}$ with $(\delta^{-1}/100,\delta)$-caps $\eta$. 
This choice is important for two reasons. On the one hand, it has area smaller than $1/2$, and this forces structure on the lattice points inside $\eta$. On the other hand, the length scale $\sim \delta^{-1}$ of $\eta$ is the smallest for which we get $L^p$ square root cancellation via small cap decoupling.
We illustrate this in Case 1 of the following four-case argument.

Decompose
$$\Phi_{\lambda,\delta}=\sum_{\eta}\Phi_{\eta},$$
with $\Phi_\eta(x)=\sum_{k\in \eta\cap \mathbb{Z}^2}e^{2\pi i k\cdot x}$.
\\
\\
Case 1. We apply (the rescaled version of) \eqref{dd3} to $f=\sum_{\eta}\Phi_\eta$ (so $\gamma=\eta$ and $f_\eta=\Phi_\eta$) and $R\sim 1/\delta$, with the sum restricted to those $\eta$ containing exactly one lattice point.

Let us check that $\eta$ has the desired length $\sim \lambda(\frac\delta\lambda)^{\beta}$, for some $\beta$ satisfying $p\le 2+\frac{2}{\beta}$. Solving $\frac1\delta=\lambda(\frac\delta\lambda)^{\beta}$ and using that $\delta=\lambda^{\frac{4}{p}-1}$, leads to $\beta=\frac{2}{p-2}$. Thus, we apply \eqref{dd3} at the critical exponent.

Note that $N_{total}\sim \lambda\delta$. We allow for the possibility that $N_{active}$ may be comparable to $N_{total}$, see the comment a few lines below. Note also that $\|f_\eta\|_{L^p(w_{B_R})}\sim R^{2/p}$. Using first the 1-periodicity of $f$, then \eqref{dd3}, we conclude with the desired bound
$$
\|\sum_{\eta}\Phi_{\eta}\|_{L^p([0,1]^2)}^p\sim R^{-2}\|\sum_{\eta}f_{\eta}\|_{L^p([0,R]^2)}^p\les (\lambda\delta)^{p/2}.
$$

The argument just presented  works when considering caps $\eta$ with $2^s\les 1$ points.
The counting arguments that we will use next show that most lattice points in $\mathcal{A}_{\lambda,\delta}$ are absorbed by such caps. In particular, for at least one of these small scales $s$, we have that  $N_{active}\ges \lambda\delta$.
This shows the sharpness of the argument, and motivates the use of small cap decoupling.
\medskip

In the remaining cases we restrict attention to those $\eta$ containing at least two lattice points. All lattice points inside $\eta$ need to sit on a line we call $l_\eta$. This is because the area of the triangle determined by any three lattice points is half an integer, while the area of (the convex hull of) $\eta$ is smaller than $1/2$. Moreover,
all lattice points on $l_\eta\cap \eta$ must be equidistant, with separation of consecutive points of order $\sim 2^m$, for some $m\ge 0$. Pigeonholing at the expense of a $(\log R)^2$ loss, we may focus on those $\eta$ corresponding to a fixed $m$, and also  containing $\sim 2^s$ lattice points, for some fixed $s\ge 0$. It follows that
\begin{equation}
\label{dd6}
\operatorname{length}(l_\eta\cap \eta)\sim 2^{s+m}\lesssim \delta^{-1}.
\end{equation}
Finally, call $\alpha$ the angle between $l_\eta$ and the long axis of $\eta$.
\\
\\
Case 2. We restrict attention to those $\eta$ with $\alpha\gg \operatorname{ecc}(\eta)\sim \delta^2$. We call them {\em active}. It is worth pointing out that $\eta$ is essentially a rectangle. In fact, each $\eta$ sits inside a $((\lambda\delta)^{1/2},\delta)$-cap (as part of $\mathcal{A}_{\lambda,\delta}$) that is also essentially a rectangle. We will call such caps by the letter $\tau$. The fact that $\tau$ is longer than $\eta$, $(\lambda\delta)^{1/2}\ge \delta^{-1}$, is a consequence of \eqref{dd7}.

\smallskip

Since $\alpha$ is much bigger than the eccentricity of $\eta$, the lines $l_\eta$ need to be different for all active $\eta$ inside a fixed $\tau$. This will force some separation between any two consecutive active $\eta_1$ and $\eta_2$, as follows. Pick two lattice points $P_1,P_2\in l_{\eta_1}\cap \eta_1$ with $\operatorname{dist}(P_1,P_2)\sim 2^m$.
Pick any point $P_3\in (l_{\eta_2}\cap \eta_2)\setminus l_{\eta_1}$. Let $d=\operatorname{dist}(P_3,l_{\eta_1})$. Since the area of $\triangle P_1P_2P_3$ is at least $1/2$, we find that
\begin{equation}
\label{dd8}
d2^m\gtrsim 1.
\end{equation}
Since $\alpha\gg \operatorname{ecc}(\eta)\sim \delta^2$, we have that
\begin{equation}
\label{dd9}
\alpha\sim\sin\alpha\sim \frac{\delta}{2^{s+m}}.
\end{equation}
On the other hand,
\begin{equation}
\label{dd10}
\alpha\sim\sin\alpha\sim \frac{d}{|P_1P_3|}.
\end{equation}
Combining these two and using that $|P_1P_3|\sim \operatorname{dist}(\eta_1,\eta_2)$ shows that
$$\operatorname{dist}(\eta_1,\eta_2)\sim d\delta^{-1}2^{s+m}.$$
When combined with \eqref{dd8}, this leads to $\operatorname{dist}(\eta_1,\eta_2)\gtrsim 2^s\delta^{-1}$.
\smallskip

This suggests partitioning each $\tau$ into $(2^s\delta^{-1},\delta)$-caps $\theta$. Each $\eta$ will sit inside some $\theta$, and each $\theta$ contains at most one active $\eta$. We call $\theta$ active if it contains some active $\eta$. Let now
$$f=\sum_{\eta:\;active}\Phi_\eta,$$
$$f_\theta=\sum_{\eta\subset\theta:\;active}\Phi_\eta.$$
We apply (the rescaled version of) \eqref{dd3} to $f$, with the caps $\gamma$ being the active $\theta's$. We need to check that the length $2^s\delta^{-1}$ of $\theta$ may be written as $\lambda(\frac{\delta}{\lambda})^{\beta}$, for some $\beta$ satisfying $p\le \min(2+\frac{2}{\beta},6)$. However, this is immediate, since we have observed earlier that $\delta^{-1}=\lambda(\frac{\delta}{\lambda})^{\frac{2}{p-2}}$. The small cap in this case is getting longer than in the previous case.
\smallskip

Periodicity and \eqref{dd3} with $R=1/\delta$ gives
$$\|\sum_{\eta:\;active}\Phi_{\eta}\|_{L^p([0,1]^2)}^p\sim R^{-2}\|\sum_{\theta:\;active}f_{\theta}\|_{L^p([0,R]^2)}^p\les R^{-2} N_{total}^{\frac{p}2-1}N_{active}\max_{\theta:active}\|f_\theta\|_{L^p(w_{B_R})}^p.$$
Note that there are $N_{total} \sim \lambda/(2^s\delta^{-1})$ caps $\theta$.
\smallskip

Here is how we evaluate the number  $N_{active}$ of active $\theta$. A $\tau$ containing at least one active $\theta$ will itself be called active. Each active $\tau$ contains $\lesssim \delta(\lambda\delta)^{1/2}/2^s$ active $\theta$, and by~\eqref{boundSms}, the number of active $\tau$ is $\lesssim (\lambda\delta)^{1/2}2^{m-s}$. We conclude that
$$N_{active}\lesssim \frac{\delta(\lambda\delta)^{1/2}}{2^s}\frac{(\lambda\delta)^{1/2}2^m}{2^s}=\lambda\delta^22^m2^{-2s}.$$
Since each $\theta$ contains $\sim 2^s$ points, we have the trivial  sharp estimate
$$\|f_\theta\|_{L^p(w_{B_R})}^p\lesssim \|f_\theta\|_{L^2(w_{B_R})}^2 2^{s(p-2)}\sim 2^{s(p-1)}R^2.$$
Putting things together, we conclude with the desired bound
$$\|\sum_{\eta:\;active}\Phi_{\eta}\|_{L^p([0,1]^2)}^p\les (\frac{\lambda\delta}{2^s})^{\frac{p}{2}-1}\lambda\delta^22^m2^{-2s}2^{s(p-1)}=(\lambda\delta)^{\frac{p}{2}}(2^{m+s}\delta)2^{s(\frac{p}{2}-3)}\lesssim (\lambda\delta)^{\frac{p}{2}},$$
where the last inequality follows from \eqref{dd6} and the fact that $p\le 6$.
\\
\\
Case 3. We now restrict attention to those $\eta$ with $(\delta/\lambda)^{1/2}\sim \operatorname{ecc}(\tau)\lesssim \alpha\lesssim \operatorname{ecc}(\eta)\sim \delta^2$. In particular, we assume $\alpha\sim 2^{-t}\delta^2$, for some fixed $t\ge 0$. The line $l_\eta$ is now the same for $\lesssim 2^t$ consecutive active $\eta$. We cover each group of such $\eta$ with a $(2^t\delta^{-1},\delta)$-cap $\sigma$, and call the common line $l_\sigma$. We call $\sigma$ active. Since $l_\sigma$ crosses at least one $\eta$, we have that $2^{m+s}\sim \delta^{-1}$.

\begin{figure}
\begin{center}
\begin{tikzpicture}[scale=2.5]

\pgfmathsetmacro{\originx}{3}
\pgfmathsetmacro{\originy}{0}
\pgfmathsetmacro{\angle}{15}

\pgfmathsetmacro{\height}{0.6}
\pgfmathsetmacro{\width}{6}

\pgfmathsetmacro{\swidth}{1.5}
\pgfmathsetmacro{\hswidth}{1}

\coordinate (origin) at (-\originx, -\originy);
\coordinate (s2origin) at (-\originx + 2*\swidth, -\originy);

\draw[blue, very thick] (origin) --++ (0, 0+\height);
\draw[name path= base][orange, ultra thick] (origin) -- ++(0 + \width, 0);
\draw[blue, very thick] (-\originx, -\originy+\height) -- ++(0+ \width, 0);
\draw[orange, ultra thick] (-\originx + \width, - \originy) -- ++ (0 , 0+\height);

\coordinate (m5) at ($(origin)+(0,-0.65)$);
\coordinate (m6) at ($(origin)+(\width, 0)+(0,-0.65)$);
\draw [decorate, decoration = {calligraphic brace, mirror}, very thick] (m5) --  (m6);
\node[below] at ($(m5)! 0.5 !(m6) +(0,-0.05)$)  {$\tau$};


\draw[red] (origin) -- ++(0, 0+\height);
\draw[red] (origin) -- ++(0 + \swidth, 0);
\draw[red] (-\originx, -\originy+\height) -- ++(0+ \swidth, 0);
\draw[red] (-\originx + \swidth, - \originy) -- ++(0, \height);

\coordinate (m1) at ($(origin)+(0,-0.3)$);
\coordinate (m2) at ($(origin)+(\swidth, -0.3)$);
\draw[gray, <->] (m1)--(m2);
\node[below] at ($(m1)! 0.5 !(m2) $)  {$2^t\delta^{-1}$};

\coordinate (m3) at ($(origin)+(0,\height)+(0,0.5)$);
\coordinate (m4) at ($(origin)+ (2*\swidth,0) +(0,\height)+(0,0.5)$);
\draw [decorate, decoration = {calligraphic brace}, very thick] (m3) --  (m4);
\node[above] at ($(m3)! 0.5 !(m4) +(0,0.05)$)  {$\theta$};

\coordinate (m7) at ($(origin)+(0,\height)+(0,0.2)$);
\coordinate (m8) at ($(origin)+ (\swidth,0) +(0,\height)+(0,0.2)$);
\draw [decorate, decoration = {calligraphic brace}, very thick] (m7) --  (m8);
\node[above] at ($(m7)! 0.5 !(m8) +(0,0.05)$)  {$\sigma_1$};

\draw[blue] (origin) -- ++(0.5*\swidth, 0)-- ++ (0,\height) -- ++ (-0.5*\swidth,0)-- cycle;
\node[below left] at ($ (origin)+(0.5*0.5*\swidth, 0.5*\height)  $) {$\eta$};

\draw[red] (s2origin) -- ++(0, 0+\height);
\draw[red] (s2origin) -- ++(0 + \swidth, 0);
\draw[red] (-\originx + 2*\swidth, -\originy+\height) -- (-\originx + 2*\swidth + \swidth, -\originy+\height);
\draw[red] (-\originx + 2*\swidth + \swidth, -\originy+\height) -- (-\originx + 2*\swidth + \swidth, -\originy);

\coordinate (m9) at ($(s2origin)+(0,\height)+(0,0.2)$);
\coordinate (m10) at ($(s2origin)+ (\swidth,0) +(0,\height)+(0,0.2)$);
\draw [decorate, decoration = {calligraphic brace}, very thick] (m9) --  (m10);
\node[above] at ($(m9)! 0.5 !(m10) +(0,0.05)$)  {$\sigma_2$};

\draw[blue] (s2origin) -- ++(0.5*\swidth, 0)-- ++ (0,\height) -- ++ (-0.5*\swidth,0)-- cycle;

\pgfmathsetmacro{\lorigin}{0.4}
\draw[ultra thick] (origin)+(\swidth, \lorigin*\height) --  ($ (180-\angle : 1.7) + (origin)+ (\swidth, \lorigin*\height ) $ );

\draw[name path=line][ultra thick] (origin)+(\swidth, \lorigin*\height) -- ($(-\angle : 3) + (origin)+(\swidth, \lorigin*\height)$) ;

\fill[red,name intersections={of=base and line,total=\t, by=inter}]
    \foreach \s in {1,...,\t}{(intersection-\s) circle (0pt) };

\pgfmathsetmacro{\radius}{0.5}
\draw[thick] (inter)+(\radius,0) arc (0: -\angle: \radius);
\node[right] at ($(inter)+ ( -0.5* \angle: \radius)$) {$\alpha$}; 

\filldraw [green] ($ (180-\angle : 0.6) + (origin)+ (\swidth, \lorigin*\height ) $ ) circle (0.7pt);
\node[below ] at ($ (180-\angle : 0.6) + (origin)+ (\swidth, \lorigin*\height ) $ ) {$P_1$}; 

\filldraw [green] ($ (180-\angle : 0.3) + (origin)+ (\swidth, \lorigin*\height ) $ ) circle (0.7pt);
\node[below ] at ($ (180-\angle : 0.3) + (origin)+ (\swidth, \lorigin*\height ) $ ) {$P_2$}; 

\filldraw [green] ($ (180-\angle : 0.9) + (origin)+ (\swidth, \lorigin*\height ) $ ) circle (0.7pt);

\filldraw [green] ($ (180-\angle : 1.2) + (origin)+ (\swidth, \lorigin*\height ) $ ) circle (0.7pt);

\draw[ultra thick] (s2origin)+(\swidth, \lorigin*\height) -- ($(180-\angle : 1.7) + (s2origin)+(\swidth, \lorigin*\height)$) ;

\draw[ultra thick] (s2origin)+(\swidth, \lorigin*\height) -- ($(-\angle : 1.6) + (s2origin)+(\swidth, \lorigin*\height)$) ;

\filldraw [green] ($(180-\angle : 1.2)+(s2origin)+(\swidth, \lorigin*\height)$) circle (0.7pt);
\node[below ] at ($(180-\angle : 1.2)+(s2origin)+(\swidth, \lorigin*\height)$) {$P_3$}; 
\filldraw [green] ($(180-\angle : 0.9)+(s2origin)+(\swidth, \lorigin*\height)$) circle (0.7pt);
\filldraw [green] ($(180-\angle : 0.6)+(s2origin)+(\swidth, \lorigin*\height)$) circle (0.7pt);
\filldraw [green] ($(180-\angle : 0.3)+(s2origin)+(\swidth, \lorigin*\height)$) circle (0.7pt);

\pgfmathsetmacro{\d}{2*\swidth * sin(\angle)}
\draw [thick, <->] (s2origin)+(\swidth, \lorigin*\height) -- ($(-90 -\angle : \d) + (s2origin)+(\swidth, \lorigin*\height)$) ;
\node[below right] at ($(s2origin)+(\swidth, \lorigin*\height)+ 0.5*(-90 -\angle : \d)$) {$d$};

\end{tikzpicture}
\end{center}
\begin{caption}{Lattice points and caps inside $\tau$}
\end{caption}
\end{figure}

Next, we prove separation between consecutive active $\sigma_1$ and $\sigma_2$, using the argument from Case 2. Pick two lattice points $P_1,P_2$ in $\sigma_1\cap l_{\sigma_1}$ with $\operatorname{dist}(P_1,P_2)\sim 2^m$ and pick a lattice point $P_3\in (\sigma_2\cap l_{\sigma_2})\setminus l_{\sigma_1}$. Letting $d=\operatorname{dist}(P_3,l_{\sigma_1})$, we find as before that $d2^m\gtrsim 1$. Also as before,
$$2^{-t}\delta^2\sim \alpha\sim \frac{d}{\operatorname{dist}(\sigma_1,\sigma_2)},$$
so $$\operatorname{dist}(\sigma_1,\sigma_2)\sim d2^t\delta^{-2}\sim d2^m2^{s+t}\delta^{-1}\gtrsim 2^{s+t}\delta^{-1}.$$

We cover each $\tau$ with $(2^{s+t}\delta^{-1},\delta)$-caps $\theta$. Each $\theta$ contains at most one active $\sigma$, and is contained in a unique  $\tau$. We call $\theta$ active if it contains some active $\sigma$, and we also call active the $\tau$ containing such $\theta$.
\smallskip

We decouple into caps $\theta$. Small cap decoupling is applicable, as $\theta$ is even  longer than in Case 2. Note first that
$$N_{total}\sim \lambda\delta 2^{-s-t}.$$
There are $\lesssim (\lambda\delta)^{1/2}2^{m-s} \sim 2^{2m} (\lambda\delta)^{1/2}\delta2^{-t}$ active $\tau$, each containing $\lesssim \frac{\delta(\lambda\delta)^{1/2}}{2^{s+t}}$ active $\theta$. Thus
the number of active $\theta$ satisfies
$$N_{active}\lesssim \lambda\delta^32^{2m-2t-s}.$$
Since $\theta$ now contains $\lesssim 2^{t+s}$ points, we have as before
$$\|f_\theta\|_{L^p(w_{B_R})}^p\lesssim \|f_\theta\|_{L^2(w_{B_R})}^2 2^{(s+t)(p-2)}\sim 2^{(s+t)(p-1)}R^2.$$
We first make the point that \eqref{dd3} is not strong enough in this case, as it leads to
\begin{align*}
\|\sum_{\eta:\;active}\Phi_{\eta}\|_{L^p([0,1]^2)}^p&\sim R^{-2}\|\sum_{\theta:\;active}f_{\theta}\|_{L^p([0,R]^2)}^p\\&\les R^{-2} N_{total}^{\frac{p}2-1}N_{active}\max_{\theta:active}\|f_\theta\|_{L^p(w_{B_R})}^p\\&\lesssim  (\lambda\delta)^{p/2}(2^{m+s}\delta)^22^{s(\frac{p}{2}-3)}2^{t(\frac{p}{2}-2)}\\&\sim (\lambda\delta)^{p/2}2^{s(\frac{p}{2}-3)}2^{t(\frac{p}{2}-2)}.
\end{align*}
When $p>4$, we cannot force this to be $\lesssim (\lambda\delta)^{p/2}$, as $t$ may be much larger than $s$. However, using instead \eqref{dd4} we find the desired upper bound
\begin{align*}
\|\sum_{\eta:\;active}\Phi_{\eta}\|_{L^p([0,1]^2)}^p&\sim R^{-2}\|\sum_{\theta:\;active}f_{\theta}\|_{L^p([0,R]^2)}^p\\&\les R^{-2} N_{total}^{3-\frac{p}2}N_{active}^{p-3}\max_{\theta:active}\|f_\theta\|_{L^p(w_{B_R})}^p\\&\lesssim  (\lambda\delta)^{p/2}2^{-(s+t)(3-\frac{p}{2})}(\delta^22^{2m-2t-s})^{p-3}2^{(s+t)(p-1)}\\&\lesssim (\lambda\delta)^{p/2}(\delta^22^{2m}2^{2s})^{p-3}2^{s(5-\frac{3p}{2})}2^{t(2-\frac{p}{2})}\\&\sim (\lambda\delta)^{p/2}2^{s(5-\frac{3p}{2})}2^{t(2-\frac{p}{2})}.
\end{align*}
This is $\lesssim (\lambda\delta)^{p/2}$ if $p\ge 4$.
\\
\\
Case 4. When $\alpha\ll \operatorname{ecc}(\tau)$, the points on distinct active $\tau$ are aligned in distinct directions. Thus, there can only be $O(2^{2m})$ active $\tau$. We have that $2^{s+m}\sim \delta^{-1}$, and each $\tau$ contains $\lesssim 2^s(\lambda\delta)^{1/2}\delta$ points. The desired bound follows from the $l^2(L^p)$ decoupling \eqref{dd11} into caps $\gamma=\tau$
\begin{align*}
\|\sum_{\eta:\;active}\Phi_{\eta}\|_{L^p([0,1]^2)}^p&\sim R^{-2}\|\sum_{\tau:\;active}f_{\tau}\|_{L^p([0,R]^2)}^p\\&\les R^{-2} N_{active}^{p/2}\max_{\tau:active}\|f_\tau\|_{L^p(w_{B_R})}^p\\&\lesssim  2^{mp}(2^s(\lambda\delta)^{1/2}\delta)^{p-1}\\&\sim (\lambda\delta)^{p/2}\frac{\delta^{-1}}{2^s\sqrt{\lambda\delta}}.
\end{align*}
This is $\lesssim (\lambda\delta)^{p/2}$ due to \eqref{dd7}.

\end{proof}
We may now prove Conjecture B in the range $\delta\ge \lambda^{-1/3}$.
\begin{cor}
\label{dd13}	
Assume $1\ge \delta\ge \lambda^{-1/3}$. Then Conjecture B holds for all $p$.	
\end{cor}
\begin{proof}
There is $p_\delta\in[4,6]$ such that $\delta=\lambda^{\frac4{p_\delta}-1}$. If $p\le p_\delta$, the result follows from Theorem \ref{dd12} and H\"older's inequality. When $p>p_\delta$, it follows from the same  theorem and the $L^\infty$ bound
$$\int|\Phi_{\lambda,\delta}|^p\le \int|\Phi_{\lambda,\delta}|^{p_\delta}(\lambda\delta)^{p-p_\delta}\les (\lambda\delta)^{p_\delta/2}(\lambda\delta)^{p-p_\delta}=\frac{(\lambda\delta)^p}{\lambda^2}.$$
\end{proof}	
A simple application of $l^2(L^6)$ decoupling gives the following.
\begin{prop}
\label{dd15}	
Conjecture B holds for $p\le 6$ and $\delta\le \lambda^{-1/3}$.
\end{prop}
\begin{proof}
We may assume $\delta\ll \lambda^{-1/3}$.
The case $p<6$ follows from $p=6$ and H\"older's inequality. When $p=6$ we use $l^2(L^6)$ decoupling \eqref{dd11}. More precisely, we decouple into $((\lambda\delta)^{1/2},\delta)$-caps $\tau$. Their volume is $\le 1/2$, so lattice points in $\tau$ are contained in a line. Reasoning as before, there are $N_{active}\les \frac{\lambda\delta}{2^{2s}}$ caps $\tau$ with $\sim 2^{2s}$ points. Then \eqref{dd11} gives
$$\int|\Phi_{\lambda,\delta}|^6\les \max_{s\ge 0}(N_{active})^{3}2^{5s}\les  \max_{s\ge 0}2^{-s}(\lambda\delta)^3\sim (\lambda\delta)^3.$$

\end{proof}
\begin{rem}
A small improvement over the results presented here, which would reach the region $\delta < \lambda^{-\frac 13}$, $p>6$ for Conjecture B, are possible with the help of bounds on exponential sums. We will not pursue this approach, in order to avoid further technical details.
\end{rem}

Theorem \ref{dd12} and thus also Corollary \ref{dd13} continue to hold true (via the same argument) if $\mathcal{A}_{\lambda,\delta}$ are the lattice points in the $\delta$-neighborhood of  $\lambda\Gamma$, where $\Gamma$ is any curve satisfying \eqref{dd14}.
However, Proposition \ref{dd15} needs the fact that $\mathcal{A}_{\lambda,\delta}$ contains $\les \lambda\delta$ lattice points. Its proof relies on this bound in order to guarantee that $N_{active}\les \lambda\delta$ for the caps $\tau$ containing only one point. For caps with at least two points, the upper bound $\les \lambda\delta/2^s$ remains true  for arbitrary $\Gamma$ as in \eqref{dd14}, via the same geometric argument, that only exploits curvature. The fact that this inequality is also true for caps with one point in the case of $S^1$ is an consequence of Lemma \ref{lemmacounting1}. 

\begin{rem} \label{remarkparabola}
For certain $\Gamma$,  the analog of the set $\mathcal{A}_{\lambda,\delta}$ may contain significantly more lattice points than $\lambda\delta$, when $\delta$ is significantly smaller than $\lambda^{-1/3}$. One such example is the parabola $\Gamma_{\mathbb{P}_1}(\xi)=\xi^2$. If $\lambda^{1/2}=n$ is an integer, then $\lambda\Gamma_{\mathbb{P}_1}$ contains $\sim \lambda^{1/2}$ lattice points $(nl,l^2), |l|\le n$, far more than $\lambda\delta$ when $\delta \ll \lambda^{-1/2}$ (this is as much as possible for a general curve, up to a subpolynomial factor, by~\cite{BombieriPila}). 

In the case of the parabola, these points are arranged along an arithmetic progression in the horizontal direction, which leads to constructive interference on a large set: note that
 $$\left|\sum_{|l|\le n}e(lnx_1+l^2x_2)\right|\sim n$$
 if $x_1\in \cup_{j\le n}[j/n,j/n+1/10n^2],\;|x_2|\le 1/10n^2$. Thus
\begin{equation}
\label{ineqparabola}
\int_{[0,1]^2} \left|\sum_{|l|\le n}e(lnx_1+l^2x_2)\right|^pdx_1dx_2\gtrsim \lambda^{\frac p2-\frac 32}.
\end{equation}

Recall that the generalized projection operator and its kernel are defined by
$$
\widetilde{P_{\lambda,\delta}} f = \sum_{k \in \mathcal{N}_\delta(\lambda \Gamma) \cap \mathbb{Z}^2}\widehat{f_k} e^{2\pi i k\cdot x} \quad \mbox{and} \quad \widetilde{\Phi_{\lambda,\delta}}(x) = \sum_{k \in \mathcal{N}_\delta(\lambda \Gamma) \cap \mathbb{Z}^2} e^{2\pi i k \cdot x}.
$$

From the inequality~\eqref{ineqparabola}, it follows that
$$
\|\widetilde{P_{\lambda,\delta}}\|_{L^2\to L^p}\gtrsim \lambda^{\frac14-\frac3{2p}} \quad \mbox{and} \quad \|\widetilde{\Phi_{\lambda,\delta}}\|_{L^p} \gtrsim \lambda^{\frac{1}{2} - \frac{3}{2p}},
$$
which shows that conjectures~\ref{conjproj} and~\ref{conjkernel} (for the latter, at least when $p$ is even) do not apply to the parabola for small enough $\delta$. We refrain from making a precise conjecture for the parabola, as there might be additional examples.
\end{rem}

\section{Graphical representation}
\label{sectiongraphical}
Figure~\ref{loriot} represents in the coordinates $(\frac{1}{p},-\frac{\log \delta}{\log p})$ the different regions where Conjecture~\ref{conjproj} is verified. The vertical coordinate gives the size of $\delta$ which decreases, making the problem harder, as one goes up in the picture; for the bottom line $\delta=1$, the conjecture corresponds to Sogge's theorem. The horizontal coordinate gives the size of $p$; if $p\leq 10$ and $p=\infty$, the conjecture is settled (with $\epsilon$ loss), but it appears that intermediate values of $p$ are harder to understand.

\begin{figure}
\centering
\begin{tikzpicture}
\begin{axis}[ xmin=0, xmax=.25, ymin=0, ymax=1 , xtick=0, ytick=0 ]
\addplot[name path=A,domain=0:1,samples=100, opacity=0] {1};
\addplot[name path=B,domain=0:1,samples=100, opacity=0] {0};
\addplot[name path=C,domain=0:1,samples=100, opacity=0] {0.33333};
\addplot[blue!50] fill between[of=B and A, soft clip={domain=0.166666:1}];
\addplot[blue!30] fill between[of=B and A, soft clip={domain=0.1:0.166666}];
\addplot[blue!30] fill between[of=B and C, soft clip={domain=0:0.1}];
\addplot [domain=0:1, samples=100, name path=g, opacity=0] {(1-6*x)/(3-2*x)};
\addplot[blue!50] fill between[of=g and B, soft clip={domain=0:1}];
\addplot [domain=0:1, samples=100, name path=g, opacity=0] {(10-64*x)/(29-14*x)};
\addplot[blue!50] fill between[of=g and B, soft clip={domain=0:1}];
\addplot [domain=0:1, samples=100, name path=f, thick, color=red!50] {1-8*x/(1+2*x)};
\end{axis}
\draw(4.6,0) node{$\bullet$};
\draw(4.6,-.5) node{$\frac{1}{6}$}; 
\draw(6.85,0) node{$\bullet$};
\draw(6.85,-.5) node{$\frac{1}{4}$}; 
\draw(7.5,-.5) node{$\frac{1}{p}$};
\draw(2.76,0) node{$\bullet$};
\draw(2.76,-.5) node{$\frac{1}{10}$}; 
\draw(0,5.67) node{$\bullet$};
\draw(-.5,5.67) node{$1$};
\draw(-2,5.67) node{$\alpha=-\frac{\log \delta}{\log \lambda}$};
\draw(0,1.89) node{$\bullet$};
\draw(-.5,1.89) node{$\frac{1}{3}$};
\end{tikzpicture}
\caption{\label{loriot}
The vertical axis corresponds to $\alpha=-\frac{\log \delta}{\log \lambda}$, and the horizontal axis to $\frac 1p$.
 In the dark blue region, Conjecture~\ref{conjproj} is verified; in the light blue region, it is verified with an $\epsilon$-loss. The red line is the curve $\delta = \lambda^{-1 + \frac{8}{p+2}}$, which separates the region where the conjecture is $\lambda^{\frac{1}{2}-\frac{2}{p}} \delta^{\frac{1}{2}}$ (below) from the region where the conjecture is $(\lambda \delta)^{\frac{1}{4}-\frac{1}{2p}}$ (above).}
\end{figure}
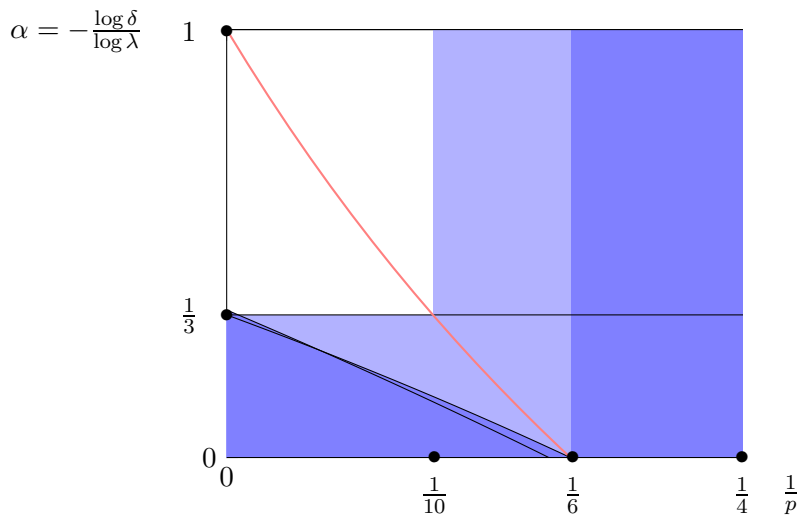

If the conjecture holds at a given point in the above diagram, then it is also true on a whole region depending on that point. These implications are summarized in the following lemma.

\begin{lem} \label{hirondelle}
\begin{itemize} \item[(i)] If Conjecture~\ref{conjproj} is satisfied at a point $(\frac{1}{p_0},\alpha_0)$ below the red curve, consider the rectangle with that point as its top right vertex. Then the conjecture holds at any point in that rectangle.
\item[(ii)] If Conjecture~\ref{conjkernel} is satisfied at a point $(\frac{1}{p_0},\alpha_0)$ above the red curve, then it holds to the right of this point, that is on the segment joining this point to $(\frac{1}{2},\alpha_0)$.
\end{itemize}
\end{lem}

\begin{proof}
\noindent $(i)$ We will show that the conjecture holds at points $(\frac{1}{p},\alpha_0)$, with $p > p_0$, and also at points $(\frac{1}{p},\alpha)$, with $\alpha < \alpha_0$; this will prove the assertion.
\begin{itemize}
\item If $p>p_0$, the Bernstein inequality gives $$\| P_{\lambda,\delta} f \|_{L^p} \lesssim \lambda^{\frac{2}{p_0}-\frac{2}{p}} \| P_{\lambda,\delta} f \|_{L^{p_0}} \lesssim \lambda^{\frac{2}{p_0}-\frac{2}{p}} \lambda^{\frac{1}{2} - \frac{2}{p_0}} \delta^{\frac{1}{2}} \| f \|_{L^2} \lesssim \lambda^{\frac{1}{2} - \frac{2}{p}} \delta^{\frac{1}{2}} \| f \|_{L^{2}}.$$
\item To deal with the case $\alpha < \alpha_0$, we observe first that the classical $TT^*$ argument shows that $\| P_{\lambda,\delta} \|_{L^2\to L^p} = \| P_{\lambda,\delta} \|_{L^{p'}\to L^p}$.
Assume now that 
$$
\| P_{\lambda_0,\delta_0} \|_{L^2 \to L^{p_0}} = C_0 \lambda_0^{\frac{1}{2} - \frac{2}{p}} \delta^{\frac{1}{2}}
$$
for constants $(\lambda_0,\delta_0,p_0,C_0)$, and consider $\delta > \delta_0$. Then the interval $(\lambda-\delta,\lambda+\delta)$ can be covered by $O(\delta / \delta_0)$ disjoint intervals $(I_j)$ of length $\delta_0$ and
$$
\| P_{\lambda_0,\delta} \|_{L^2 \to L^p_0}^2 \leq \Big\| \sum_j P_{I_j} \Big\|_{L^2 \to L^{p_0}}^2 = \Big\| \sum_j P_{I_j} \Big\|_{L^{p_0'} \to L^{p_0}} \leq \sum_j \left\|P_{I_j} \right\|_{L^{p_0'} \to L^{p_0}} \lesssim \frac{\delta}{\delta_0} \lambda_0^{1-\frac{4}{p_0}} \delta_0 \lesssim\lambda_0^{1-\frac{4}{p_0}} \delta.
$$
\end{itemize}

\medskip 
\noindent
$(ii)$ is a consequence of interpolation, and of the trivial bound $\| P_{\lambda,\delta} \|_{L^2 \to L^2} \lesssim 1$: if $2 < p < p_0$, choosing $\theta$ such that $\frac 1 p - \frac 1 2 = \theta \left( \frac 1 {p_0} - \frac{1}{2} \right)$,
$$
\| P_{\lambda,\delta} \|_{L^2 \to L^p} \lesssim \| P_{\lambda,\delta} \|_{L^2 \to L^2}^{1-\theta} \| P_{\lambda,\delta} \|_{L^2 \to L^{p_0}}^\theta \lesssim (\lambda \delta)^{\theta \left( \frac 1 4 - \frac 1 {2p_0} \right)} = (\lambda \delta)^{\frac 1 4 - \frac 1 {2p}}.
$$
\end{proof}

We now turn to Conjecture~\ref{conjkernel}; Figure~\ref{pivert} depicts the different regions where it is verified. Furthermore, if the conjecture holds at a given point in this diagram, then it follows on a region depending on that point. Such implications are summarized in the following lemma.

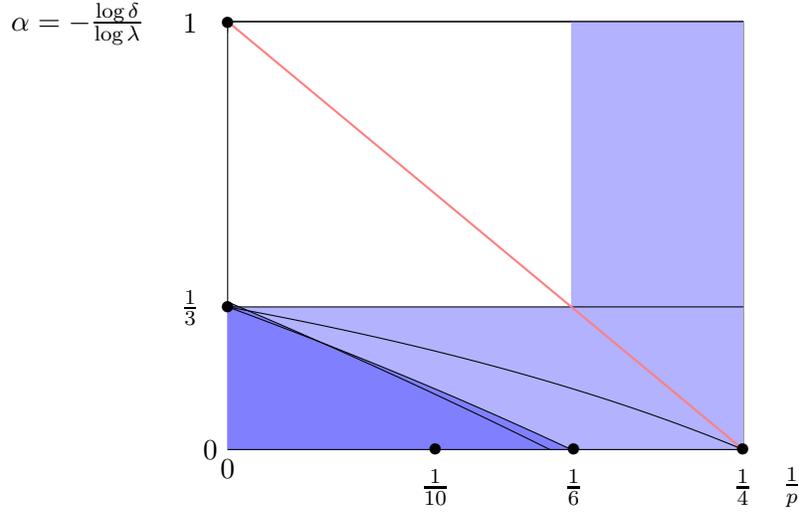
\begin{figure}
\centering
\begin{tikzpicture}
\begin{axis}[
   		xmin=0, xmax=.25,
   		ymin=0, ymax=1,
   		xtick=0, ytick=0 ]
\addplot[name path=A,domain=0:1,samples=100, opacity=0] {1};
\addplot[name path=B,domain=0:1,samples=100, opacity=0] {0};
\addplot[name path=C,domain=0:1,samples=100, opacity=0] {0.33333};
\addplot[blue!30] fill between[of=B and A, soft clip={domain=0.166666:0.25}];
\addplot[blue!30] fill between[of=B and C, soft clip={domain=0:0.166666}];
\addplot [domain=0:0.25, samples=100, name path=H, opacity=0] {(1-4*x)/(3-4*x)};
\addplot [domain=0:0.25, samples=100, name path=g, opacity=0] {(1-6*x)/(3-2*x)};
\addplot[blue!50] fill between[of=g and B, soft clip={domain=0:0.25}];
\addplot [domain=0:0.25, samples=100, name path=g, opacity=0] {(10-64*x)/(29-14*x)};
\addplot[blue!50] fill between[of=g and B, soft clip={domain=0:0.25}];
\addplot [domain=0:0.25, samples=100, name path=f, thick, color=red!50] {1-4*x};
\end{axis}
\draw(4.6,0) node{$\bullet$};
\draw(4.6,-.5) node{$\frac{1}{6}$}; 
\draw(6.85,0) node{$\bullet$};
\draw(6.85,-.5) node{$\frac{1}{4}$}; 
\draw(7.5,-.5) node{$\frac{1}{p}$};
\draw(2.76,0) node{$\bullet$};
\draw(2.76,-.5) node{$\frac{1}{10}$}; 
\draw(0,5.67) node{$\bullet$};
\draw(-.5,5.67) node{$1$};
\draw(-2,5.67) node{$\alpha=-\frac{\log \delta}{\log \lambda}$};
\draw(0,1.89) node{$\bullet$};
\draw(-.5,1.89) node{$\frac{1}{3}$};
\end{tikzpicture}
\caption{\label{pivert}
The vertical axis corresponds to $\alpha=-\frac{\log \delta}{\log \lambda}$, and the horizontal axis to $\frac 1p$.
 In the dark blue region, Conjecture~\ref{conjkernel} is verified without loss; in the light blue region, it is verified with an $\epsilon$-loss. The red line is the curve $\delta = \lambda^{-1 + \frac{4}{p}}$, which separates the region where the conjecture is $\lambda^{1-\frac{2}{p}} \delta$ (below) from the region where the conjecture is $(\lambda\delta)^{\frac{1}{2}}$ (above).}
\end{figure}

\begin{lem}
\begin{itemize} \item[(i)] If Conjecture~\ref{conjkernel} is satisfied at a point $(\frac{1}{p_0},\alpha_0)$ below the red curve, consider the rectangle with that point as its top right vertex. Then the conjecture holds at any point in that rectangle.
\item[(ii)] If Conjecture~\ref{conjkernel} is satisfied with $\epsilon$ loss at a point $(\frac{1}{p_0},\alpha_0)$ above the red curve, then it holds with $\epsilon$ loss to the right of this point, that is on the segment joining this point to $(\frac{1}{2},\alpha_0)$.
\end{itemize}
\end{lem}

\begin{proof} This is very similar to Lemma~\ref{hirondelle}. For $(i)$, we use  Bernstein's inequality and the fact that, if $\delta_0 < \delta$, then $\Phi_{\lambda,\delta}$ can be written as the sum of $O(\frac{\delta}{\delta_0})$ functions of the type $\Phi_{\lambda,\delta_0}$. For $(ii)$, it suffices to interpolate with the trivial bound $\| \Phi_{\lambda,\delta} \|_{L^2} \lesssim \lambda^\epsilon (\lambda \delta)^{\frac 12}$.
\end{proof}

\bibliographystyle{abbrv}
\bibliography{references}

\end{document}